\documentclass[12pt]{article}
\usepackage{amsmath,amsthm,amssymb, color}
\input amssym.def
\input amssym.tex
\date{}
\newtheorem{proposition}{Proposition}
\newtheorem{theorem}{THEOREM}
\newtheorem{corollary}{Corollary}
\newtheorem{lemma}{Lemma}
\newtheorem{remark}{Remark}

\begin{document}

\title{{\bf Interpolation of Ces{\`a}ro sequence and function spaces}}

\author {Sergey V. Astashkin\thanks{Research partially supported by
RFBR grant no. 12-01-00198-a.} \,{\small and} Lech Maligranda}

\date{}

\maketitle

\renewcommand{\thefootnote}{\fnsymbol{footnote}}

\footnotetext[0]{2000 {\it Mathematics Subject Classification}:
46E30, 46B20, 46B42}
\footnotetext[0]{{\it Key words and phrases}: Ces{\`a}ro sequence and function spaces, weighted 
Ces{\`a}ro function spaces, Copson sequence and function spaces, interpolation, K-functional, 
K-method of interpolation}

\vspace{-9 mm}

\begin{abstract}
\noindent {\footnotesize The interpolation property of Ces{\`a}ro sequence and function spaces is investigated.
It is shown that $Ces_p(I)$ is an interpolation space between $Ces_{p_0}(I)$ and $Ces_{p_1}(I)$
for $1 < p_0 < p_1 \leq \infty$ and $1/p = (1 - \theta)/p_0 + \theta /p_1$ with $0 < \theta < 1$, where
$I = [0, \infty)$ or $[0, 1]$. The same result is true for Ces{\`a}ro sequence spaces. On the other hand,
$Ces_p[0, 1]$ is not an interpolation space between $Ces_1[0, 1]$ and $Ces_{\infty}[0, 1]$.

}
\end{abstract}

\begin{center}
{\bf 1. Introduction and preliminaries}
\end{center}

Structure of the Ces{\`a}ro sequence and function spaces was investigated by several authors (see, for
example, \cite{Be}, \cite{MPS} and \cite{A}, \cite{AM} and references therein). Here we are interested in studying 
interpolation properties of Ces{\`a}ro sequence
and function spaces. The main purpose is to give interpolation theorems for the Ces{\`a}ro sequence
spaces $ces_p$ and Ces{\`a}ro function spaces $Ces_{p}(I)$ on $I = [0, \infty)$ and $I = [0, 1]$. In the
case of $I = [0, \infty)$ some interpolation results for Ces{\`a}ro function spaces are contained implicitely
in \cite{MS}. Moreover, using the so-called $K^+$-method of interpolation it was proved in \cite{CFM}
that Ces{\`a}ro sequence space $ces_p$ is an interpolation space with respect to the couple $(l_1, l_1(2^{-k}))$.
Our main aim is to give a rather complete description of Ces{\`a}ro spaces as interpolation spaces with respect
to appropriate couples of weighted $L_1$-spaces as well as Ces{\`a}ro spaces. For example, if either $I = [0, \infty)$ 
or $[0, 1]$ and $1 < p_0 < p_1 \leq \infty$ with $\frac{1}{p} = \frac{1 - \theta}{p_0} + \frac{\theta}{p_1}$ for $0 < \theta < 1$,
then
\begin{equation} \label{equation1}
\left( Ces_{p_0}(I), Ces_{p_1}(I) \right)_{\theta, p} = Ces_p(I) ~~{\rm and} ~~ \left(ces_{p_0},
ces_{p_1} \right)_{\theta, p} = ces_p,
\end{equation}

\vspace{-2mm}

\noindent
where $(\cdot, \cdot)_{\theta, p}$ denotes the K-method of interpolation.

We have a completely different situation in a more interesting
and non-trivial case when $I = [0, 1]$ and $p_0 = 1,$ $p_1 = \infty$. It turns out that $Ces_p[0, 1]$ is not an 
interpolation space between the spaces $Ces_1[0, 1]$ and
$Ces_{\infty}[0, 1]$ and $\left( Ces_1[0, 1], Ces_{\infty}[0, 1] \right)_{\theta, p}$ for $1 < p < \infty$ is a
weighted Ces{\`a}ro function space.

Let us collect some necessary definitions and notations related to the interpolation theory of operators
as well as Ces{\`a}ro, Copson and down spaces.

For two normed spaces $X$ and $Y$ the symbol $X \stackrel {C} \hookrightarrow Y$ means that
the imbedding $X \subset Y$ is continuous with the norm which is not bigger than C, i.e.,
$\|x\|_{Y} \leq C \|x\|_{X}$ for all $x \in X$, and $X \hookrightarrow Y$ means that $X \stackrel {C} \hookrightarrow Y$
for some $C > 0$. Moreover, $X = Y$ means that $X \hookrightarrow Y$ and $Y \hookrightarrow X$, that is, the spaces 
are the same and the norms are equivalent.  If $f$ and $g$ are real functions, then the symbol $f \approx g$  means 
that $c^{-1}\, g \leq f \leq c\, g$ for some $c \geq 1$.

For a Banach couple $\bar{X} = (X_0, X_1)$ of two compatible Banach spaces $X_0$ and $X_1$ consider two Banach
spaces $X_0 \cap X_1$ and $X_0 + X_1$ with its natural norms
\begin{equation*}
 \| f \|_{X_0 \cap X_1} = \max \,( \| f \|_{X_0},  \| f \|_{X_1}) ~~ {\rm for} ~~ f \in  X_0 \cap X_1,
\end{equation*}
\vspace{-2mm}
and
\begin{equation*}
 \| f \|_{X_0 + X_1} = \inf \, \{ \| f_0 \|_{X_0}, + \| f_1 \|_{X_1}: f = f_0 + f_1, f_0 \in X_0, f_1 \in X_1\} ~~ {\rm for} ~~ f \in  X_0 + X_1.
\end{equation*}

\vspace{-2mm}

\noindent
For more careful definitions of a Banach couple, intermediate and interpolation spaces with some results introduced briefly below, see \cite[pp. 91-173, 289-314, 338-359]{BK} and \cite[pp. 95-116]{BS}.

A Banach space $X$ is called an {\it intermediate space} between $X_0$ and $X_1$ if $X_0 \cap X_1  \hookrightarrow X  \hookrightarrow X_0 + X_1$. Such a space $X$ is called an {\it interpolation space} between $X_0$ and $X_1$
(we write: $X\in Int(X_0, X_1)$) if, for any bounded linear operator $T: X_0 + X_1 \rightarrow X_0 + X_1$ such that the
restriction $T_{| X_i}: X_i \rightarrow X_i$ is bounded for $i = 0, 1$, the restriction $T_{|_X}: X \rightarrow X$ is also bounded and $\| T\|_{X \rightarrow X} \leq C \, \max \, \{ \| T\|_{X_0 \rightarrow X_0} , \| T\|_{X_1 \rightarrow X_1} \}$ for some $C \geq 1$. If $C = 1$, then $X$ is called an {\it exact interpolation space} between $X_0$ and $X_1$.

An {\it interpolation method} or {\it interpolation functor} ${\cal F}$ is a construction (a rule) which assigns to every Banach couple $\bar{X} = (X_0, X_1)$ an interpolation space ${\cal F}(\bar{X})$ between $X_0$ and $X_1$. The interpolation functor ${\cal F}$ is called {\it exact} if the space
${\cal F}(\bar{X})$ is an exact interpolation space for every Banach couple $\bar{X}$.
One of the most important interpolation methods is the {\it $K$-method} known also as the {\it real Lions-Peetre interpolation method}. For a Banach couple
$\bar{X} = (X_0, X_1)$ the {\it Peetre K-functional} of an element $f \in X_0+X_1$ is defined for $t > 0$ by
$$
K(t, f; X_0, X_1) = \inf \{ \| f_0\|_{X_0} + t \| f_1\|_{X_1}: f = f_0 + f_1, f_0 \in X_0, f_1 \in X_1 \}.
$$
Then the {\it spaces of the K-method of interpolation} are
$$
(X_0, X_1)_{\theta, p} = \{ f \in X_0+X_1: \| f \|_{\theta, p} = \Big(\int_0^{\infty} [t^{-\theta} K(t, f; X_0, X_1)]^p \frac{dt}{t}\Big)^{1/p}
< \infty \}
$$
if $0 < \theta < 1$ and $1 \leq p < \infty$, and
$$
(X_0, X_1)_{\theta, \infty} = \{ f \in X_0+X_1: \| f \|_{\theta, \infty} = \sup_{t > 0} \frac {K(t, f; X_0, X_1)}{t^{\theta} } < \infty \}
$$
if $0 \leq \theta \leq 1$. Very useful in calculations is the so-called {\it reiteration formula} showing the stability of the $K$-method
of interpolation. If $1 \leq p_0, p_1, p \leq \infty, 0 < \theta_0, \theta_1, \theta < 1$ and $\theta_0 \neq \theta_1$, then with
equivalent norms
\begin{equation} \label{equation2}
\big( (X_0, X_1)_{\theta_0, p_0},  (X_0, X_1)_{\theta_1, p_1} \big)_{\theta, p} =  (X_0, X_1)_{\eta, p},
\end{equation}
where $\eta = (1 - \theta) \theta_0 + \theta \theta_1$ (see \cite[Theorem 2.4, p. 311]{BS}, \cite[Theorems 3.5.3]{BL},  \cite[Theorem 3.8.10]{BK}) and \cite[Theorem 1.10.2]{Tr}).

The space $(X_0, X_1)_{\Phi}^K$ of the {\it general $K$-method of interpolation}, where $\Phi$ is a {\it parameter of the 
$K$-method}, i.e., a Banach function space  on $((0, \infty), dt/t)$ containing the function $t \longmapsto \min \{1, t\}$, is 
the Banach space of all $f \in X_0 + X_1$ such that $K(\cdot, f; X_0, X_1) \in \Phi$ with the norm 
$\| f\|_{K_{\Phi}} = \| K(\cdot, f; X_0, X_1) \|_{\Phi}$. The space $(X_0, X_1)_{\Phi}^K$ is an exact interpolation space 
between $X_0$ and $X_1$.

In particular, if $L_p=L_p(\Omega,\mu),$ where $(\Omega,\mu)$ is a complete $\sigma$-finite measure space, then 
for any $f \in L_1 + L_{\infty}$ we have
\begin{equation}
\label{K-functional}
K(t, f; L_1, L_{\infty}) = \int_0^t \, f^*(s)\, ds.
\end{equation}

\vspace{-2mm}

\noindent
Here and next $f^*$ denotes the non-increasing rearrangement of $|f|$ defined by
$f^*(s) = \inf \{\lambda > 0: \mu(\{x \in \Omega: |f(x)| > \lambda \}) \leq s \} $ (see \cite[Proposition 3.1.1]{BK}, \cite[pp. 78-79]{KPS}, \cite[Theorem 6.2, pp. 74-75]{BS}). Moreover, for two non-negative weight functions $w_0, w_1$ and for $f \in L_1(w_0) + L_1(w_1)$ we have
\begin{equation}
\label{K-functional 1}
K(t, f; L_1(w_0), L_1(w_1)) = \| \min (w_0, t w_1 ) \, f\, \|_{L_1}
\end{equation} (see \cite[Proposition 3.1.17]{BK} and \cite[p. 391]{Ov}).

If the inequality $K(t, g; X_0, X_1) \leq K(t, f; X_0, X_1)$ $(t > 0)$ 
with $f \in X$ and $g \in X_0 + X_1$ implies that $g \in X$
and $\| g \|_X \leq C \, \| f \|_X$ for arbitrary $X\in Int(X_0,X_1)$ and some $C \geq 1$ independent of $X$, $f$ and $g$, 
then $(X_0, X_1)$ is called a
{\it $K$-monotone} or {\it Calder\'on-Mityagin couple}. For arbitrary $K$-monotone couple $(X_0, X_1)$
the spaces $(X_0, X_1)_{\Phi}^K$ of the general $K$-method are the only interpolation spaces between $X_0$ and $X_1$ (see \cite{BK}).

Now, to treat interpolation results for Ces{\`a}ro spaces we need to define these spaces.
The {\it Ces{\`a}ro sequence spaces} $ces_p$ are the sets of real sequences $x = \{x_k \}$
such that
$$
\| x \|_{ces(p)} = \left [ \sum_{n=1}^{\infty} \left (\frac{1}{n} \, \sum_{k=1}^n |x_k| \right)^p \right ]^{1/p} < \infty,
~~~{\rm for} ~~~1 \leq p < \infty,
$$
and
$$
\| x \|_{ces(\infty)} = sup_{n \in \mathbb N} \, \frac{1}{n} \sum_{k=1}^n |x_k| < \infty,  ~~~ {\rm for} ~~~ p = \infty.
$$
The {\it Ces{\`a}ro function spaces} $Ces_p = Ces_p (I)$ are the classes of Lebesgue measurable real functions
$f$ on $I = [0,1]$ or $I = [0,\infty)$ such that
$$
\|f\|_{Ces(p)} = \left [ \int_I \left (\frac{1}{x}\int_{0}^{x} |f(t)| \, dt \right)^{p} ~dx \right]^{1/p} < \infty,
 ~~~{\rm for} ~~~1 \leq p < \infty,
$$
and
$$
\|f\|_{Ces(\infty)} = \sup_{0 < x \in I} ~\frac{1}{x}\int_{0}^{x}
|f(t)|~dt < \infty, ~~~ {\rm for} ~~~ p = \infty.
$$

The Ces{\`a}ro spaces are Banach lattices which are not symmetric except as they are trivial, namely,
$ces_1 = \{0\}, Ces_1[0, \infty) = \{0\}$. By a {\it symmetric space} we
mean a Banach lattice $X$ on $I$ satisfying the additional property: if $g^*(t)=f^*(t)$ for all $t>0$,
$f \in X$ and $g \in L^0(I)$ (the set of all classes of Lebesgue measurable real functions on $I$) then $g \in X$ 
and $\| g \|_X = \| f\|_X$ (cf. \cite{BS}, \cite{KPS}).
Moreover, $l_p  \stackrel {p^{\prime}} \hookrightarrow ces_p, L_p(I)  \stackrel {p^{\prime}} \hookrightarrow Ces_p(I)$ for $ 1 < p \leq \infty$
(in what follows $\frac{1}{p}+\frac{1}{p'}=1$),
and if $1 < p < q < \infty$, then $ces_p  \stackrel {1} \hookrightarrow ces_q  \stackrel {1} \hookrightarrow ces_{\infty}$. Also for $I = [0, 1]$ and $1 < p < q < \infty$ we have $L_{\infty} \stackrel {1} \hookrightarrow  Ces_{\infty} \stackrel {1} \hookrightarrow Ces_q \stackrel {1} \hookrightarrow Ces_p \stackrel {1} \hookrightarrow Ces_1 = L_1(\ln 1/t)$ and $Ces_{\infty} \stackrel {1} \hookrightarrow L_1$.

Let $1 \leq p < \infty$. The {\it Copson sequence spaces} $cop_p$ are the sets of real sequences $x = \{x_k \}$
such that
$$
\| x \|_{cop(p)} = \left [ \sum_{n=1}^{\infty} \left ( \sum_{k=n}^{\infty} \frac{|x_k|}{k} \right)^p \right ]^{1/p} < \infty,
$$
and the {\it Copson function spaces} $Cop_p = Cop_p (I)$ are the classes of Lebesgue measurable real functions
$f$ on $I = [0,\infty)$ or $I = [0,1]$ such that
$$
\|f\|_{Cop(p)} = \left [ \int_0^{\infty} \left ( \int_{x}^{\infty} \frac{|f(t)|}{t} \, dt \right)^{p} ~dx \right]^{1/p} < \infty
 ~~~{\rm for} ~~~I = [0, \infty),
$$
and
$$
\|f\|_{Cop(p)} =  \left [ \int_0^1 \left (\int_{x}^1 \frac{|f(t)|}{t} \, dt \right)^{p} ~dx \right]^{1/p} < \infty
 ~~~{\rm for} ~~~I = [0, 1].
$$
Sometimes we will use the {\it Ces{\`a}ro operators}
$C_d\,x(n) = \frac{1}{n} \sum_{k=1}^n |x_k|$, $Cf(x) = \frac{1}{x}\int_{0}^{x} |f(t)|\, dt$
and the {\it Copson operators}
$C_d^*x(n) = \sum_{k=n}^{\infty} \frac{|x_k|}{k}$, $C^*f(x) = \int_{(x, \infty) \cap I} \frac{|f(t)|}{t}\, dt$
related to introduced spaces.
Then $ces_p$ (resp.  $cop_p$) consists of all real sequences
$x = \{x_k \}$ such that $C_d\,x \in l_p$ (resp. $C^*_d\,x \in l_p$) and $Ces_p (I)$
(resp. $Cop_p (I)$) consists of all classes of Lebesgue measurable real functions $f$ on $I$ such that
$Cf \in L_p(I)$ ( resp. $C^*f \in L_p(I)$) with natural norms.
By the Copson inequalities (cf. \cite[Theorems 328 and 331]{HLP}, \cite[p. 25]{Be} and \cite[p. 159]{KMP}),
which are valid for $1 \leq p < \infty,$ we have: $\| C^*_d\,x \|_{l_p} \leq p\, \| x \|_{l_p}$ for
$x \in l_p$ and $\| C^* f \,\|_{L_p(I)} \leq p\, \| f \,\|_{L_p(I)} $ for $f \in L_p(I).$
Therefore, $l_p \stackrel {p} \hookrightarrow cop_p, L_p  \stackrel {p} \hookrightarrow Cop_p$.

We can define similarly the spaces $cop_{\infty}$ and $Cop_{\infty}$ but then it is easy to see that $cop_{\infty} = l_1(1/k)$
and $Cop_{\infty} = L_1(1/t)$. Moreover, for $I = [0, 1]$ we have $L_p \stackrel {p} \hookrightarrow Cop_p \stackrel {1} \hookrightarrow
Cop_1 = L_1$.

We will consider also more general Ces{\`a}ro spaces $Ces_E (I)$, where $E$ is a Banach function space on
$I$ with the natural norm $\| f\|_{Ces(E)} = \| Cf \|_E$.

For a Banach function space $E$ on $I = [0, \infty)$ the {\it down space} $E^{\downarrow}$ is the
collection of all $f \in L^0$ such that the norm
$$
\| f \|_{E^{\downarrow}} = \sup \int_I |f(t)| g(t) \, dt < \infty,
$$

\vspace{-2mm}

\noindent
where the supremum is taken over all non-negative, non-increasing Lebesgue measurable functions $g$
from the K\"othe dual $E'$ of $E$ such that $ \| g \|_{E^{\prime}} \leq 1$.
Let us remind that the K\"othe dual of a Banach function space $E$ is defined as
$$
E^{\prime} = \{ f \in L^0: \| f \|_{E^{\prime}} =  \sup_{\| g \|_E \leq 1} \int_I |f(t) g(t)| \, dt < \infty \}.
$$

\vspace{-2mm}

\noindent
It is routine to check that the space $E^{\downarrow}$ has the Fatou property, that is, if $0\le f_n$ increases
to $f$ a.e. on $I$ and $\sup_{n\in\mathbb N}\|f_n\|_{E^{\downarrow}}<\infty,$ then $f\in E^{\downarrow}$ and $\|f_n\|_{E^{\downarrow}}$
increases to $\|f\|_{E^{\downarrow}}.$ Moreover, $E''  \stackrel {1} \hookrightarrow E^{\downarrow}$, where $E''$ is the second K\"othe dual of $E$.
Recall also that a Banach function space $E$ has the Fatou property if and only if $E = E''$ with the equality of the norms.

Sinnamon (\cite{Si01}, Theorem 3.1) proved that if $E$ is a symmetric space on $I = [0, \infty)$,
then $\| f \|_{E^{\downarrow}} \approx \| Cf \|_E$ if and only if the Ces{\`a}ro operator $C: E \rightarrow E$ is
bounded. In particular, then $E^{\downarrow} = Ces_E$. Moreover, $(L_1)^{\downarrow} = L_1$ since
$$
\| f \|_{L_1^{\downarrow}} = \sup_{0 \leq g} \frac{\int_0^{\infty} |f(t)| g(t)\, dt }{\| g\|_{L_{\infty}}} \geq \sup_{0 \leq g \downarrow}
\frac{\int_0^{\infty} |f(t)| g(t)\, dt }{\| g\|_{L_{\infty}}} \geq \frac{\int_0^{\infty} |f(t)| \, dt }{\| 1 \|_{L_{\infty}}} = \| f \|_{L_1}
$$
(cf. \cite{MS}, p. 194).

The paper is organized as follows. In Section 2 we proved that the
Ces{\`a}ro and Copson sequence and function spaces on $[0, \infty)$ are interpolation spaces obtained by the $K$-method from weighted $L_1$-spaces. At the same time, in the case of $I = [0, 1]$, the $K$-method gives only the Copson spaces as interpolation spaces with respect to analogous couple of weighted $L_1$-spaces (see Theorem 1(iii)). In particular, we obtain a new description of the interpolation spaces $(L_1, L_1(1/t))_{1-1/p, p}$ in off-diagonal case both for $I = [0,\infty)$ and $I = [0, 1]$.

In Section 3 it is shown that the Ces{\`a}ro function spaces $Ces_p[0, \infty), 1 < p< \infty$ can be obtained by the $K$-method of interpolation also from the couple $(L_1[0, \infty), Ces_{\infty}[0, \infty))$. Hence, applying the reiteration theorem, we conclude that $Ces_p[0, \infty)$ are interpolation spaces  with respect to the couple $( Ces_{p_0}[0, \infty), Ces_{p_1}[0, \infty))$ for arbitrary $1 < p_0 < p_1 \leq \infty$ and $\frac{1}{p} = \frac{1-\theta}{p_0} + \frac{\theta}{p_1}$ with $0 < \theta < 1$.

In Section 4 the interpolation of Ces{\`a}ro function spaces on the segment $[0, 1]$ is investigated. We prove that for $1 < p < \infty$
\begin{equation*}
\left( L_1(1-t)[0, 1], Ces_{\infty}[0,1] \right)_{\theta, p} = Ces_p[0, 1] ~~{\rm with} ~~ \theta = 1 - 1/p.
\end{equation*}
As a consequence of this result and reiteration equality (\ref{equation2}), we infer
\begin{equation} \label{equation3}
\left(Ces_{p_0}[0, 1], Ces_{p_1}[0,1] \right)_{\theta, p} = Ces_p[0, 1]
\end{equation}
for all $1 < p_0 < p_1 \leq \infty$ and  $\frac{1}{p} = \frac{1-\theta}{p_0} + \frac{\theta}{p_1}$ with $0 < \theta < 1$.

Our important part of interest is to look on interpolation spaces between the spaces $Ces_1[0, 1]$ and $Ces_{\infty}[0, 1]$.
In Section 5, in Theorem 3, we find an equivalent expression for the $K$-functional with respect to this couple and then in Section 6 we proved that
the real interpolation spaces $(Ces_1[0, 1], Ces_{\infty}[0, 1])_{1-1/p, p}$ for $1 < p< \infty$ can be identified with the weighted
 Ces{\`a}ro function spaces $Ces_p(\ln e/t)[0, 1]$.

 Finally, in Section 7, we show in Theorem 6 that $Ces_p[0, 1]$ for $1 < p < \infty$ are not interpolation spaces between the spaces $Ces_1[0, 1]$ and $Ces_{\infty}[0, 1]$.

\vspace{3 mm}
\begin{center}
{\bf 2. Ces{\`a}ro and Copson spaces as interpolation spaces with respect to weighted $L_1$-spaces}
\end{center}

We start with the main result in this part.

\begin{theorem} \label{Thm1} (i) If $1 < p < \infty$, then
\begin{equation*}
(l_1, l_1(1/k) )_{1-1/p, p} = ces_p = cop_p.
\end{equation*}
\begin{itemize}
\item[$(ii)$] If $I = [0, \infty)$ and $1 < p < \infty$, then
$$
(L_1, L_1(1/t) )_{1-1/p, p} = Ces_p = Cop_p.
$$
\item[$(iii)$] If $I = [0, 1]$ and $1 < p \leq \infty$, then
$$
(L_1, L_1(1/t) )_{1-1/p, p} = Cop_p.
$$
Moreover, $Cop_p  \stackrel {p^{\prime}} \hookrightarrow Ces_p$
and the reverse imbedding does not hold.
\end{itemize}
\end{theorem}

\begin{proof} (i) If $f \in l_1+l_1(1/k)$, then
$K(t, x; l_1, l_1(1/k)) = t \, \sum_{k=1}^{\infty} \frac{|x_k|}{k}$ for $0 < t \leq 1$,
and
$$
K(t, x; l_1, l_1(1/k)) = \sum_{k=1}^{\infty} |x_k| \min(1, \frac{t}{k}) =
\sum_{k=1}^{[t]} |x_k| + t \sum_{k= [t] + 1} ^{\infty} \frac{|x_k|}{k}.
$$
for $t \geq 1$. Therefore, for $n \leq t < n+1 ~(n \geq 1)$, we have
$$
\frac{K(t, x; l_1, l_1(1/k))}{t} \leq \frac{1}{n} \sum_{k=1}^n |x_k| + \sum_{k= n + 1} ^{\infty} \frac{|x_k|}{k}
= C_d x(n) + C_d^*\,x(n+1).
$$
Since
\begin{eqnarray*}
C_d C_d^*\,x(n)
&=&
\frac{1}{n} \sum_{m=1}^n \Big(\sum_{k=m}^{\infty} \frac{|x_k|}{k}\Big) =
\frac{1}{n} \left[ \sum_{k =1}^n \Big(\sum_{m=1}^k \frac{|x_k|}{k}\Big) + \sum_{k=n+1}^{\infty} \Big(\sum_{m=1}^n \frac{|x_k|}{k}\Big)
\right]\\
&=&
\frac{1}{n} \sum_{k=1}^n |x_k| + \sum_{k= n + 1} ^{\infty} \frac{|x_k|}{k} = C_d x(n) + C_d^*\,x(n+1),
\end{eqnarray*}
it follows that, for $n \leq t < n+1 ~(n \geq 1)$,
$$
\frac{K(t, x; l_1, l_1(1/k))}{t} \leq C_d C_d^*\,x(n).
$$
Using the classical Hardy inequality (cf. \cite[Theorem 326]{HLP} or \cite[Theorem 1]{KMP}), we obtain
\begin{eqnarray*}
\| x \|_{1-1/p, p}
&=&
\left( \int_0^\infty \Big(\frac{K(t, x; l_1, l_1(1/k))}{t}\Big)^p\, dt \right)^{1/p}\\
&=&
\left[ C_d^*\,x(1)^p + \sum_{n=1}^{\infty} \int_n^{n+1} \Big(\frac{K(t, x)}{t}\Big)^p\, dt  \right]^{1/p}\\
&\leq&
\left[ C_d^*\, x(1)^p + \sum_{n=1}^{\infty} \left( C_d C_d^*\,x(n) \right)^p \right]^{1/p}\\
&\leq &
C_d^*\,x(1) + \, \| C_dC_d^* \,x \|_{l_p} \leq C_d^* \,x(1) + p^{\prime}\, \| C_d^* \,x \|_{l_p}\\
&\leq&
(p^{\prime} +1) \| C_d^* \,x \|_{l_p} = (p^{\prime} +1) \| x \|_{cop(p)}.
\end{eqnarray*}
This means that $cop_p \hookrightarrow (l_1, l_1(1/k) )_{1-1/p, p}$.
On the other hand, for $n \leq t < n+1 ~(n \geq 1)$, we have
$$
\frac{K(t, x; l_1, l_1(1/k))}{t} \geq \sum_{k=n+1}^{\infty} \frac{|x_k|}{k} = C_d^* \,x(n+1)
$$
and
\begin{eqnarray*}
\| x \|_{1-1/p, p}
&=&
\left( \int_0^\infty \Big(\frac{K(t, x; l_1, l_1(1/k))}{t}\Big)^p\, dt \right)^{1/p}\\
&\geq&
\left( C_d^*\,x(1)^p + \sum_{n=1}^{\infty} C_d^* \,x(n+1)^p \right)^{1/p} = \| C_d^* \,x \|_{l_p} = \| x \|_{cop(p)},
\end{eqnarray*}
which gives the reverse imbedding $ (l_1, l_1(1/k) )_{1-1/p, p}  \stackrel {1} \hookrightarrow cop_p$.
The equality of the spaces $ces_p = cop_p$ for $ 1 < p < \infty$ was proved by Bennett (cf. \cite{Be}, Theorems
4.5 and 6.6).
\vspace{2mm}

(ii) For $f \in L_1+L_1(1/s) = L_1(\min(1, 1/s))$ we have
$$
K(t, f; L_1, L_1(1/s)) = \int_0^{\infty} |f(s)| \min(1, t/s)\, ds = \int_0^t |f(s)|\, ds +
t \int_t^{\infty} \frac{|f(s)|}{s}\, ds.
$$
Thus,
$$
\frac{K(t, f; L_1, L_1(1/s))}{t} = Cf(t) + C^*f(t),\;t>0,
$$
and therefore
\begin{equation}
\label{extra formula}
\| f \|_{1-1/p, p} = \left( \int_0^\infty \Big(\frac{K(t, f; L_1, L_1(1/s))}{t}\Big)^p\, dt \right)^{1/p} =
\| Cf + C^*f \|_{L_p(0, \infty)}.
\end{equation}
Since, by Fubini theorem,
\begin{eqnarray*}
C^*Cf(t)
&=&
\int_t^{\infty}\Big( \frac{1}{u^2} \int_0^u |f(s)|\, ds\Big) \, du \\
&=&
\int_0^t \Big(\int_t^{\infty} \frac{1}{u^2} du\Big) |f(s)| \, ds + \int_t^{\infty} \Big(\int_s^{\infty} \frac{1}{u^2} du\Big) |f(s)|\, ds\\
&=&
\frac{1}{t} \int_0^t |f(s)|\, ds + \int_t^{\infty} \frac{|f(s)|}{s}\, ds = Cf(t) + C^*f(t),
\end{eqnarray*}
from the Copson inequality (cf. \cite[Theorem 328]{HLP}) it follows that
\begin{eqnarray*}
\| f \|_{Ces(p)}
&=&
\| C f \|_{L_p(0, \infty)} \leq \| Cf + C^*f \|_{L_p(0, \infty)} \\
&=&
\| C^* Cf  \|_{L_p(0, \infty)} \leq p \| Cf  \|_{L_p(0, \infty)}
=  p  \| f  \|_{Ces(p)}.
\end{eqnarray*}
Combining this with \eqref{extra formula}, we obtain $\| f \|_{1-1/p, p} \approx  \| f  \|_{Ces(p)}.$

On the other hand, since
\begin{eqnarray*}
CC^*f(t)
&=&
\frac{1}{t} \int_0^t \Big(\int_u^{\infty} \frac{|f(s)|}{s} \, ds\Big) \, du \\
&=&
\frac{1}{t} \int_0^t \Big(\int_0^s du\Big) \frac{|f(s)|}{s} \, ds + \frac{1}{t} \int_t^{\infty} \Big(\int_0^t du\Big) \frac{|f(s)|}{s}\, ds\\
&=&
\frac{1}{t} \int_0^t |f(s)|\, ds + \int_t^{\infty} \frac{|f(s)|}{s}\, ds = Cf(t) + C^*f(t),
\end{eqnarray*}
then, by Hardy inequality,
\begin{eqnarray*}
\| f \|_{Cop(p)}
&=&
\| C^* f \|_{L_p(0, \infty)} \leq \| Cf + C^*f \|_{L_p(0, \infty)} \\
&=&
\| C C^*f  \|_{L_p(0, \infty)} \leq p^{\prime} \| C^*f  \|_{L_p(0, \infty)}
=  p^{\prime} \| f  \|_{Cop(p)},
\end{eqnarray*}
and,  applying \eqref{extra formula} once more, we conclude that $\| f \|_{1-1/p, p} \approx  \| f  \|_{Cop(p)}.$

(iii) For $I = [0, 1]$ and $f \in L_1+L_1(1/s) = L_1$ we have
$K(t, f; L_1, L_1(1/s)) = \| f \|_1 $ for $t \geq 1$
and
$$
K(t, f; L_1, L_1(1/s)) = \int_0^t |f(s)| \, ds + t \int_t^1 \frac{|f(s)|}{s} \, ds = t Cf(t) + t C^*f(t)
$$
for $0 < t \leq 1$. Therefore, for $1 < p < \infty$
\begin{eqnarray*}
\| f \|_{1-1/p, p}
&=&
\left( \int_0^1 [ Cf(t) + C^*f(t) ]^p\, dt + \int_1^{\infty} t^{-p} \| f \|_1^p\, dt \right)^{1/p}\\
&=&
\left( \| Cf + C^*f \|_p^p + \frac{1}{p-1} \| f \|_1^p \right)^{1/p}.
\end{eqnarray*}
Firstly, the last expression is not smaller than $\| C^*\,f \|_p = \| f \|_{Cop(p)}$. On the other hand, since again
$CC^*f(t) = Cf(t) + C^*f(t), $ by Hardy inequality, it follows that
\begin{eqnarray*}
\| f \|_{1-1/p, p}
&=&
\left( \| C C^*f \|_p^p + \frac{1}{p-1} \| f \|_1^p \right)^{1/p}  \leq
\| C C^*f \|_p + (p-1)^{-1/p} \| f \|_1 \\
&\leq&
p^{\prime} \| C^*f \|_p + (p-1)^{-1/p} \| f \|_{Cop(p)} =(p^{\prime} + (p-1)^{-1/p}) \,  \| f \|_{Cop(p)}.
\end{eqnarray*}
Thus, $(L_1, L_1(1/t) )_{1-1/p, p} = Cop_p$ with equivalent norms for $1 < p < \infty$. For $p = \infty$
we have $(L_1, L_1(1/t) )_{1, \infty} = L_1(1/t)=Cop_{\infty}[0, 1]$.

The imbedding  $Cop_p  \stackrel {p^{\prime}} \hookrightarrow Ces_p$ for $ 1 <p  \leq \infty$ follows from
the inequality
$$
\| f \|_{Ces(p)} = \| Cf \|_p \leq \| Cf + C^*f \|_p = \| C C^*f \|_p \leq p^{\prime} \| C^*f \|_p = p^{\prime} \| f\|_{Cop(p)}.
$$
Moreover, $Ces_p[0, 1] \cap L_1[0, 1]  \stackrel {p+1} \hookrightarrow Cop_p[0, 1]$ for
$ 1 \leq p < \infty$. In fact, observe that in the case of $I = [0, 1]$ the composition operator $C^* C$ has an additional term.
More precisely,
$$
C^*C f(t) = Cf(t) + C^*f(t) -\int_0^1 |f(s)|\, ds.
$$
Therefore,
\begin{eqnarray*}
\| f \|_{Cop(p)}
&=&
\| C^* f \|_p \leq \| Cf + C^*f \|_p \\
&=&
\| C^* Cf + \int_0^1 |f(s)| \, ds \|_p \leq \| C^* C f \|_p + \| f \|_1\\
&\leq&
p \,  \| C f \|_p  + \| f\|_1 \leq (p + 1) \, \max( \| f\|_{Ces(p)}, \| f\|_1).
\end{eqnarray*}

Finally, let us show that $Ces_p \not \hookrightarrow Cop_p$ by comparing norms of the functions
$f_h(t) = \frac{1}{\sqrt{1-t}} \chi_{[h, 1)} (t), 0 < h < 1$ in these spaces. We have
\begin{equation*}
C^*(f_h)(t) = \left\{
\begin{array}{ll}
\int_h^1 \frac{1}{s \sqrt{1-s}}\, ds, & ~~~\mbox{{\rm if} ~$0 < t \leq h$,} \\
\int_t^1 \frac{1}{s \sqrt{1-s}} \, ds, &
~~~\mbox{{\rm if} ~$h \leq t \leq 1$,}\\
\end{array}
\right.
\end{equation*}
and
\begin{eqnarray*}
\| f_h\|_{Cop(p)}^p
&=&
\| C^* (f_h)\|_p^p \geq \int_0^h \Big(\int_h^1 \frac{1}{s\sqrt{1-s}}\, ds\Big)^p\, dt\\
&=&
h \, \Big(\int_h^1 \frac{1}{s\sqrt{1-s}}\, ds\Big)^p \geq  h\, \Big(\int_h^1 \frac{1}{\sqrt{1-s}}\, ds\Big)^p\\
&=&
2^p h (1-h)^{p/2}.
\end{eqnarray*}
Also,
\begin{equation*}
C(f_h)(t) = \left\{
\begin{array}{ll}
0, & ~~~\mbox{{\rm if} ~$0 < t \leq h$,} \\
\frac{2}{t}(\sqrt{1-h} - \sqrt{1-t}), &
~~~\mbox{{\rm if} ~$h \leq t \leq 1$,}\\
\end{array}
\right.  \label{def_phicomp}
\end{equation*}

\vspace{-2mm}

\noindent
and
\begin{eqnarray*}
\| f_h\|_{Ces(p)}^p
&=&
\| C (f_h)\|_p^p = 2^p \int_h^1 \Big(\frac{\sqrt{1-h} -\sqrt{1-t}}{t}\Big)^p\, dt\\
&\leq&
2^p \int_h^1 \frac{(1-h)^{p/2}}{t^p}\, dt = 2^p (1-h)^{p/2} \, \frac{1- h^{p-1}}{(p-1) h^{p-1}}.
\end{eqnarray*}

\vspace{-2mm}

\noindent
Thus
\vspace{-2mm}
$$
\frac{\| f_h\|_{Cop(p)}^p}{\| f_h\|_{Ces(p)}^p} \geq \frac{2^p h (1-h)^{p/2} (p-1) h^{p-1}}{2^p (1-h)^{p/2} (1- h^{p-1})}
= (p-1) \, \frac{h^p}{1-h^{p-1}} \rightarrow \infty ~~{\rm as} ~ h \rightarrow 1^+,
$$
and the proof is complete. \end{proof}

\begin{remark} \label{Remark1} {\rm Alternatively, the space $ces_p$ for $1 < p < \infty$ can be obtained as
an interpolation space with respect to the couple $(l_1, l_1(2^{-n}))$ by the so-called $K^+$-method being
a version of the standard K-method, precisely,
$ces_p = (l_1, l_1(2^{-n}))_{l_p(1/n)}^{K^+}$ (cf. \cite[the proof of Theorem 6.4]{CFM}). }
\end{remark}

\begin{remark} \label{Remark2} {\rm The results in Theorem \ref{Thm1} give a description of the real
interpolation spaces $(L_1, L_1(1/t))_{1-1/p, p}$ in the off-diagonal case. Before it was only known that
they are intersections of weighted $L_1(w)$-spaces with the weights $w$ from certain sets
(cf. \cite[Theorem 4.1]{Gi}, \cite[Theorem 2]{MP}) or some block spaces (cf. \cite[Lemma 3.1]{AKMNP}).}
\end{remark}

The following corollary follows directly from Theorem \ref{Thm1}, reiteration formula (\ref{equation2})
and the equalities $Cop_\infty[0,1]=L_1(1/t)$ and $Cop_1[0,1]=L_1.$

\begin{corollary} \label{Corollary1} If $1 < p_0 < p_1 < \infty$ and $\frac{1}{p} = \frac{1-\theta}{p_0} + \frac{\theta}{p_1}$
with $ 0 < \theta < 1$, then
\begin{equation} \label{equation4}
(ces_{p_0}, ces_{p_1})_{\theta, p} = ces_p, (Ces_{p_0}[0, \infty), Ces_{p_1}[0, \infty))_{\theta, p} = Ces_p[0, \infty).
\end{equation}
If $1 \leq p_0 < p_1 \leq \infty$ and $\frac{1}{p} = \frac{1-\theta}{p_0} + \frac{\theta}{p_1}$ with $ 0 < \theta < 1$,
then
\begin{equation} \label{equation5}
(Cop_{p_0}[0, 1], Cop_{p_1}[0, 1])_{\theta, p} = Cop_p[0, 1].
\end{equation}
\end{corollary}

\vspace{-3mm}

\begin{remark} \label{Remark3} {\rm  Another proof of the second equality for the spaces on $[0, \infty)$
from the last corollary was given by Sinnamon \cite[Corollary 2]{Si91}.}
\end{remark}

\vspace{2 mm}
\begin{center}
{\bf 3. Ces{\`a}ro spaces on $[0, \infty)$ as interpolation spaces
with respect to the couple $(L_1, Ces_{\infty})$}
\end{center}

All the spaces considered in this part are on the interval $I = [0, \infty)$.
By \cite[p. 194]{MS} the down space $D^{\infty}: = (L_{\infty})^{\downarrow} = Ces_{\infty}$ isometrically.
On the other hand, for a Banach lattice $F$ with the Fatou property we have
$F \in Int(L_1, D^{\infty}) = Int(L_1, Ces_{\infty})$
if and only if $F = E^{\downarrow}$ with equality of norms for some $E \in Int(L_1, L_{\infty})$
(see \cite[Theorem 6.4]{MS}). Then, in particular, $L_p^{\downarrow} \in Int(L_1, Ces_{\infty})$.
Since the operator $C$ is bounded in $L_p$ for $ 1 < p \leq \infty,$ by \cite[Theorem 3.1]{Si01}, it follows that
$$
\| f \|_{L_p^{\downarrow} } = \| \,|f|\, \|_{L_p^{\downarrow} } \approx \| C f \|_{L_p} = \| f \|_{Ces(p)}.
$$
Thus, for any $ 1 < p < \infty$ we have $Ces_p \in Int(L_1, Ces_{\infty})$ and $Ces_p = L_p^{\downarrow}$.
\vspace{2mm}
Moreover, the following more precise and general assertion, which is
an almost immediate consequence of Theorem 6.4 from \cite{MS}, holds.

\begin{proposition} \label{Proposition1} Let $E, F \in Int(L_1, L_{\infty})$ and $\Phi$ be an interpolation
Banach lattice with respect to the couple $(L_{\infty}, L_{\infty}(1/u))$ on $(0, \infty)$. Then we have
\begin{equation} \label{equation6}
(E^{\downarrow}, F^{\downarrow})_{\Phi}^K = [(E, F)_{\Phi}^K]^{\downarrow}.
\end{equation}
In particular, if $1 < p < \infty$, then
\begin{equation} \label{equation7}
(L_1, Ces_{\infty})_{1-1/p, p} = Ces_p.
\end{equation}
\end{proposition}

\begin{proof} Firstly, since the Banach couple $(L_1,L_\infty)$ is $K$-monotone \cite[Theorem~2.4.3]{KPS},
by the assumption and the Brudny{\u \i}-Krugljak theorem (cf. \cite[Theorem~4.4.5]{BK}), $E = (L_1, L_{\infty})_{\Phi_0}^K$
and $ F = (L_1, L_{\infty})_{\Phi_1}^K$ with some interpolation Banach lattices $\Phi_0$ and $\Phi_1$
with respect to the couple $(L_{\infty}, L_{\infty}(1/u))$ on $(0, \infty)$. Applying the reiteration
theorem for the general $K$-method (see \cite[Theorem 3.3.11]{BK}), we obtain
$$
(E, F)_{\Phi}^K = ( (L_1, L_{\infty})_{\Phi_0}^K,  (L_1, L_{\infty})_{\Phi_1}^K)_{\Phi}^K =  (L_1, L_{\infty})_{\Psi}^K,
$$
where $\Psi = (\Phi_0, \Phi_1)_{\Phi}^K$. Moreover, from the proof of Theorem 6.4 in \cite{MS} and the equality
$L_1^{\downarrow}=L_1$ (see Section 1) it follows
$$
E^{\downarrow} = [(L_1, L_{\infty})_{\Phi_0}^K]^{\downarrow} = (L_1, D^{\infty})_{\Phi_0}^K, \, F^{\downarrow} = [(L_1, L_{\infty})_{\Phi_1}^K]^{\downarrow} = (L_1, D^{\infty})_{\Phi_1}^K
$$
and
$$
[(E, F)_{\Phi}^K]^{\downarrow} = [(L_1, L_{\infty})_{\Psi}^K]^{\downarrow} = (L_1, D^{\infty})_{\Psi}^K.
$$
Therefore, using the reiteration theorem once again, we obtain
\begin{eqnarray*}
(E^{\downarrow}, F^{\downarrow})_{\Phi}^K = ((L_1, D^{\infty})_{\Phi_0}^K,  (L_1, D^{\infty})_{\Phi_1}^K)_{\Phi}^K
= (L_1, D^{\infty})_{\Psi}^K =  [(E, F)_{\Phi}^K]^{\downarrow}.
\end{eqnarray*}
and equality (\ref{equation6}) is proved. In particlular, from (\ref{equation6}) and the well-known identification formula
$(L_1, L_{\infty})_{1-1/p, p}=L_p$ \cite[Theorem~5.2.1]{BL} it follows that
$$
(L_1, Ces_{\infty})_{1-1/p, p} = (L_1^{\downarrow}, L_{\infty}^{\downarrow})_{1-1/p, p} = L_p^{\downarrow} = Ces_p.
$$
and also equality (\ref{equation7}) is proved.
\end{proof}

For a given symmetric space $E$ on $I = [0, \infty)$ the Ces{\`a}ro function space $Ces_E$ is defined by the norm
$\| f\|_{Ces(E)} = \| Cf \, \|_E$. If operator $C$ is bounded in $E$, then, by \cite[Theorem 3.1]{Si01},
$Ces_E = E^{\downarrow}$. Therefore, applying Proposition 1, we obtain

\begin{corollary} \label{Corollary2} Let the operator $C$ be bounded in symmetric spaces $E$ and $F$ on $[0,\infty)$
and let $\Phi$ be an interpolation Banach lattice with respect to the couple $(L_{\infty}, L_{\infty}(1/u))$ on
$(0, \infty)$. Then
\begin{equation*}
(Ces_E, Ces_F)_{\Phi}^K = Ces_{(E, F)_{\Phi}^K}.
\end{equation*}
In particular, for arbitrary $1 < p_0 < p_1 \leq \infty$
\begin{equation} \label{equation8}
(Ces_{p_0}, Ces_{p_1})_{\theta, p} = Ces_p, ~{\it where} ~~0 < \theta < 1~~{\it and}~~\frac{1}{p} = \frac{1-\theta}{p_0} + \frac{\theta}{p_1}.
\end{equation}
\end{corollary}

\begin{remark} \label{Remark 4} {\rm If $1 < p < \infty$, then the restriction of the space $Ces_p[0, \infty)$ to the interval
$[0, 1]$ coincides with the intersection $Ces_p[0, 1] \cap L_1[0, 1]$ (cf. \cite{AM}, Remark 5). Therefore, if we ``restrict" formula
(\ref{equation8}) to $[0, 1]$ we obtain only
\begin{equation*}
(Ces_{p_0}[0,1] \cap L_1[0, 1], Ces_{p_1}[0,1] \cap L_1[0, 1])_{\theta, p} = Ces_p[0,1] \cap L_1[0, 1],
\end{equation*}
where $1<p_0<p_1<\infty$ and $\frac{1}{p} = \frac{1-\theta}{p_0} + \frac{\theta}{p_1}$.}
\end{remark}

\begin{center}
{\bf 4. Ces{\`a}ro spaces on $[0, 1]$ as interpolation spaces
with respect to the couple $(L_1(1-t), Ces_{\infty})$}
\end{center}

In contrast to the case of the semi-axis $[0, \infty)$, $Ces_p[0, 1]$ for $1 \leq p < \infty$ is not
even an intermediate space between $L_1[0, 1]$ and $Ces_{\infty}[0, 1]$.
In fact, $Ces_{\infty}[0, 1] \hookrightarrow L_1[0, 1]$,
but it easy to show that $Ces_p[0, 1] \not \subset L_1[0, 1]$ for every $1 \leq p < \infty$.

On the other hand, from the inequality $1 - u \leq \ln 1/u$ $(0< u \leq 1)$ it follows that
$Ces_{p}[0, 1],$ $1 \le p < \infty,$ is an intermediate space between the spaces
$L_1(1-t)[0, 1]$ and $Ces_{\infty}[0, 1]$, because of
\begin{equation*} \label{7}
Ces_{\infty}[0, 1]  \stackrel {1} \hookrightarrow  Ces_p[0, 1] \stackrel {1} \hookrightarrow Ces_1[0, 1] = L_1(\ln1/t)[0, 1]
\stackrel {1} \hookrightarrow L_1 (1-t)[0, 1].
\end{equation*}

\begin{theorem} \label{Thm2} If $1 < p < \infty$, then
\begin{equation} \label{imbedding9}
(Ces_1[0, 1],  Ces_{\infty}[0, 1])_{1-1/p, p}  \stackrel {1} \hookrightarrow Ces_p[0, 1]
\end{equation}
and
\begin{equation} \label{equality10}
(L_1(1-t)[0, 1], Ces_{\infty}[0, 1])_{1-1/p, p} = Ces_p[0, 1].
\end{equation}
\end{theorem}

\begin{proof}
All function spaces in this proof are considered on the segment $I = [0, 1]$ if it is not indicated
something different.

At first, for any $f \in Ces_1$ and all $0 < t \leq 1$ we have
\begin{equation} \label{estimate11}
K(t,f):=K(t, f; Ces_1, Ces_{\infty}) \geq \int_0^t (Cf)^*(s) \, ds.
\end{equation}
In fact, we can assume that $f \geq 0$. If $f = g + h, g \geq 0, h \geq 0, g \in Ces_1, h \in Ces_{\infty}$,
then $Cf = Cg + Ch$ and, therefore, by formula \eqref{K-functional},
\begin{eqnarray*}
\| g \|_{Ces(1)} + t \, \| h \|_{Ces(\infty)}
&=&
\| C g \|_{L_1} + t \, \| Ch \|_{L_{\infty}}\\
&\geq&
\inf \{ \| y \|_{L_1} + t \, \| z \|_{L_{\infty}}: Cf = y + z, y \in L_1, z \in L_{\infty} \} \\
&=&
K(t, Cf; L_1, L_{\infty}) = \int_0^t (Cf)^*(s)\, ds.
\end{eqnarray*}

\vspace{-2mm}

Taking the infimum over all suitable $g$ and $h$ we get (\ref{estimate11}). Next, by the definition of the
real interpolation spaces, we obtain
\begin{eqnarray*}
\| f \|_{1-1/p, p}^p
&\ge&
\int_0^1 \left[t^{1/p-1} K(t, f)\right]^p \frac{dt}{t} = \int_0^1 t^{-p} K(t, f)^p \, dt \\
&\geq&
\int_0^1 t^{-p} \left[\int_0^t (Cf)^*(s)\, ds\right]^p \, dt \geq \| Cf \|_{L_p[0, 1]}^p = \| f \|_{Ces(p)}^p,
\end{eqnarray*}
and the proof of imbedding (\ref{imbedding9}) is complete.
\vspace{2mm}

Before to proceed with the proof of (\ref{equality10}) we introduce the following notation:
for a Banach function space $E$ on $I=[0,\infty)$ or $[0,1]$ and any set $A\subset I$ by $E|_{A}$ we will mean
the subspace of $E$, which consists of all functions $f$ such that ${\rm supp} \, f \subset A$.
Let us denote also $X_p:=(L_1(1-t), Ces_{\infty})_{1-1/p, p}.$
Since
$$
\|f\|_{X_p} \approx \|f \chi_{[0, 1/2]}\|_{X_p} + \|f \chi_{[1/2,1]}\|_{X_p}.
$$
then for proving (\ref{equality10}) it is sufficient to check that
\begin{equation}\label{equation12}
 \|f \chi_{[0, 1/2]}\|_{X_p} \approx \|f \chi_{[0, 1/2]}\|_{Ces_p}
\end{equation}
and
\begin{equation}\label{equation13}
 \|f \chi_{[1/2,1]}\|_{X_p} \approx \|f \chi_{[1/2,1]}\|_{Ces_p}.
\end{equation}
Firstly, since $L_1(1-t)|_{[0,1/2]}=L_1[0,\infty)|_{[0,1/2]}$ and
$Ces_\infty|_{[0,1/2]}=Ces_\infty[0,\infty)|_{[0,1/2]},$ then, by Proposition \ref{Proposition1}
(see formula \eqref{equation7}), we obtain
\begin{equation}\label{equation14}
 \|f \chi_{[0, 1/2]}\|_{X_p} \approx \|f \chi_{[0, 1/2]}\|_{(L_1[0,\infty), Ces_{\infty}[0,\infty))_{1-1/p, p}}
\approx \|f \chi_{[0, 1/2]}\|_{Ces_p[0,\infty)}.
\end{equation}
Note that
\begin{equation} \label{equation15}
Ces_p[0, \infty)|_{[0,1/2]} = Ces_p[0,1]|_{[0,1/2]}
\end{equation}
with equivalence of norms. In fact, by  \cite[Remark 5]{AM}, $Ces_p[0, \infty)|_{{[0, 1]}} = Ces_p \cap L_1.$
If ${\rm supp} \, g \subset [0, 1/2]$, then we have
\begin{eqnarray*}
\| g \|_{L_1} = \int_0^{1/2} |g(s)| \, ds \leq 2^{1/p} \, \Big( \int_{1/2}^1 (\frac{1}{t} \int_0^t |g(s)| \, ds)^p \, dt \Big)^{1/p}
\leq 2^{1/p} \, \| g \|_{Ces(p)}.
\end{eqnarray*}
Combining this together with the previous equality, we obtain (\ref{equation15}).
From (\ref{equation15}) and \eqref{equation14} it follows \eqref{equation12}.

Now, we prove \eqref{equation13}.
Since $(L_1(1-s)|_{[1/2,1]}, Ces_{\infty}|_{[1/2,1]})$ is a complemented subcouple of the Banach couple
$(L_1(1-s), Ces_{\infty}),$ then, by the well-known result of Baouendi and Goulaouic
\cite[Theorem 1]{BG} which is valid for arbitrary interpolation method
(see also \cite[Theorem 1.17.1]{Tr}), we have
$$
\|f \chi_{[1/2,1]}\|_{X_p} \approx \|f \chi_{[1/2,1]}\|_{Y_p},
$$
where $Y_p:=(L_1(1-s)|_{[1/2,1]}, Ces_{\infty}|_{[1/2,1]})_{1-1/p, p}.$
Therefore, \eqref{equation13} will be proved whenever we show that
\begin{equation}\label{equation16}
Y_p= Ces_p|_{[1/2,1]}.
\end{equation}

On the one hand, since $1 - u \leq \ln1/u \leq 2(1-u)$ for all
$1/2 \leq u \leq 1$ and $Ces_1 = L_1(\ln1/s)$, then $Ces_1|_{[1/2,1]}=L_1(1-s)|_{[1/2,1]},$ and, by already proved imbedding \eqref{imbedding9},
we obtain
$$
Y_p=(Ces_1|_{[1/2,1]}, Ces_{\infty}|_{[1/2,1]})_{1-1/p, p}\subset Ces_p|_{[1/2,1]}.
$$

For proving the opposite imbedding we note, firstly, that for any function $h$ with
${\rm supp} \, h \subset [1/2, 1]$  we have
$$
\| h \|_{Ces(\infty)} = \sup_{1/2 \leq x \leq 1} \frac{1}{x} \int_{1/2}^x |h(s)| \, ds,
$$
whence
$$
\| h \|_{L_1} =  \int_{1/2}^1 |h(s)| \, ds \leq  \| h \|_{Ces(\infty)} \leq 2 \,  \int_{1/2}^1 |h(s)| \, ds = 2\| h \|_{L_1}.
$$
Therefore, using formula for
the $K$-functional with respect to a couple of weighted $L_1$-spaces (see \eqref{K-functional 1}), we obtain
\begin{equation} \label{equation17}
G(t, h) \leq K(t, h; L_1(1-s)|_{[1/2,1]}, Ces_{\infty}|_{[1/2,1]}) \leq 2\, G(t, h),
\end{equation}
where
$$
G(t, h) = K(t, h; L_1(1-s)|_{[1/2,1]}, L_1|_{[1/2,1]})=\int_{1/2}^1 \min(1-s, t) |h(s)| \, ds.$$
Furthermore, let $h \in L_1|_{[1/2,1]}$. Then
$$
Ch(s) = \frac{1}{s}  \int_{1/2}^s |h(u)| \, du \geq  \int_{1/2}^s |h(u)| \, du,
$$
whence
$$
(Ch)^*(s) \geq  \int_{1/2}^{1-s} |h(u)| \, du, ~ 0 < s \leq 1.
$$
Therefore, for all $0 \leq t \leq 1$, we obtain
\begin{eqnarray*}
\int_0^t (Ch)^*(s)\, ds
& \geq&
\int_0^t \Big( \int_{1/2}^{1-s} |h(u)| \, du\Big) \, ds\\
&=&
\int_{1/2}^{1-t} \Big( \int_0^t |h(u)| \, ds\Big) \, du + \int_{1-t}^1\Big( \int_0^{1-u} |h(u)| \, ds\Big) \, du \\
&=&
t\, \int_{1/2}^{1-t}  |h(u)| \, du + \int_{1-t}^1 (1-u)|h(u)| \, du =G(t,h).
\end{eqnarray*}
From this inequality and the definition of $G(t,h)$ it follows that the estimate
$$
\int_0^{\min(1, t)} (Ch)^*(s)\, ds \geq  G(t,h)
$$
holds for all $t > 0$.
Hence, by (\ref{equation17}) and Hardy classical inequality for every $h \in Ces_p$ with ${\rm supp} \, h \subset [1/2, 1],$
we have
\vspace{-2mm}
\begin{eqnarray*}
\| h \|_{Y_p}
&=&
\Big ( \int_0^{\infty} t^{-p} K(t, h; L_1(1-s)|_{[1/2,1]}, Ces_{\infty}|_{[1/2,1]})^p \, dt \Big)^{1/p} \\
&\leq&
2 \, \Big ( \int_0^{\infty} t^{-p} G(t, h)^p \, dt \Big)^{1/p}\le
2 \,  \Big ( \int_0^{\infty} t^{-p} \Big( \int_0^{\min(1,t)} (Ch)^*(s) \, ds\Big)^p \, dt \Big)^{1/p}\\
&\le&
2 \,  \Big [ \Big( \int_0^1 t^{-p} \Big( \int_0^t (Ch)^*(s) \, ds\Big)^p \, dt \Big)^{1/p} +   \Big( \int_1^{\infty} t^{-p}
\Big( \int_0^1 (Ch)^*(s) \, ds\Big)^p \, dt \Big)^{1/p} \Big] \\
&\leq&
2 \, \Big[ \frac{p}{p-1}\, \| Ch\|_{L_p[0, 1]} + \frac{1}{(p-1)^{1/p}}\, \| Ch \|_{L_1[0, 1]} \Big]
\leq
\frac{4p}{p-1}\, \| h\|_{Ces_p[0, 1]}.
\end{eqnarray*}
Thus, $Ces_{p}|_{[1/2,1]}\subset Y_p,$ equality (\ref{equation16}) holds, and the proof is complete.
\end{proof}

The following result is an immediate consequence of equality (\ref{equality10}) and the reiteration equality
(\ref{equation2}).

\begin{corollary} \label{Corollary3}
If $1 < p_0 < p_1 \leq \infty$ and $\frac{1}{p} = \frac{1-\theta}{p_0} + \frac{\theta}{p_1}$ with $ 0 < \theta < 1$,
then
\begin{equation*}
(Ces_{p_0}[0, 1], Ces_{p_1}[0, 1])_{\theta, p} = Ces_p[0, 1].
\end{equation*}
\end{corollary}

\begin{remark} \label{Remark new}
{\rm An inspection of the proof of Theorem \ref{Thm2} shows that
\begin{equation*}
(Ces_1|_{[1/2,1]}, Ces_{\infty}|_{[1/2,1]})_{1-1/p, p}=Ces_p|_{[1/2,1]}.
\end{equation*}
for every $1<p<\infty$ with equivalence of norms.}
\end{remark}

\begin{remark} \label{Remark 5} {\rm Comparison of formulas from Remark 4 and Corollary \ref{Corollary3}
shows that the real method $(\cdot, \cdot)_{\theta, p}$ ``well" interpolates the intersection of Ces{\`a}ro spaces on the
segment $[0,1]$ with the space $L_1[0, 1]$ or, more precisely, we have
\begin{equation*}
(Ces_{p_0}[0,1] \cap L_1[0, 1], Ces_{p_1}[0,1] \cap L_1[0, 1])_{\theta, p} = (Ces_{p_0}[0,1], Ces_{p_1}[0,1])_{\theta, p}  \cap L_1[0, 1],
\end{equation*}
for all $1 < p_0 < p_1 \leq \infty, 0 < \theta < 1$ and $\frac{1}{p} = \frac{1-\theta}{p_0} + \frac{\theta}{p_1}$.}
\end{remark}

\begin{remark} \label{Remark 5new}
{\rm We will see futher that imbedding \eqref{imbedding9} is strict for every $1<p<\infty$ and, even more,
that $Ces_p[0,1]$ is not an interpolation space between the spaces $Ces_1[0,1]$ and $Ces_\infty[0,1].$
Thus, the weighted $L_1$-space $L_1(1-t)[0,1]$ is in a sense the ``proper" end of the scale of Ces{\`a}ro
spaces $Ces_p[0,1],$ $1<p\le\infty.$}
\end{remark}

\newpage
\begin{center}
{\bf 5. The $K$-functional for the couple $(Ces_1[0, 1], Ces_{\infty}[0, 1])$}
\end{center}
In this part we will find an equivalent expression for the $K$-functional
$$
K(t, f) = K(t, f; Ces_1, Ces_{\infty}) = K(t, f; Ces_1[0, 1], Ces_{\infty}[0, 1]).
$$
We start with a lemma giving its lower estimate. Let us introduce two functions defined on $(0, 1]$
by formulas
\begin{equation} \label{equation18}
\tau_1(t) = t/\ln(e/t) ~~ {\rm and} ~~ \tau_2(t) = e^{-t} ~ {\rm for} ~~  0 < t \leq 1.
\end{equation}
It is easy to see that there exists a unique $t_0 \in (0, 1)$ such that $\tau_1(t_0) = \tau_2(t_0)$ and
$\tau_1(t) < \tau_2(t)$ if and only if $0 < t < t_0$.

\vspace{3mm}
\begin{lemma} \label{lemma1}{\rm (lower estimates).} Let $f \in Ces_1[0, 1],$ $f \geq 0$ and $0 < t \leq 1.$
\begin{itemize}
\item[$(i)$] If $f_0 = f \chi_{[0, \tau_1(t)] \cup [\tau_2(t), 1]}$, then
\begin{equation} \label{19}
K(t, f) \geq \frac{1}{4}\, \| f_0 \|_{Ces(1)}.
\end{equation}
\item[$(ii)$] If $f_1 = f \chi_{[\tau_1(t), \tau_2(t)]}$, then
\begin{equation} \label{20}
K(t, f) \geq \frac{1}{e^2}\, t\, \| f_1 \|_{Ces(\infty)}.
\end{equation}
\end{itemize}
\end{lemma}

\begin{proof} (i) Firstly, let us prove that
\begin{equation} \label{21}
K(t, f) \geq \frac{1}{3}\, \|  f \chi_{[0, \tau_1(t)]} \|_{Ces(1)} ~~ {\it for ~ all} ~~ 0 < t \leq 1.
\end{equation}
Let $f\in Ces_1,$ $f = g + h$, where $g \in Ces_1, h \in Ces_{\infty}.$ We may assume that $f\ge 0$ and $0 \leq g \leq f,
0 \leq h \leq f$. Then
\begin{eqnarray}
3 \, ( \| g \|_{Ces(1)} + t \| h \|_{Ces(\infty)} )
&\geq&
\| g \|_{Ces(1)} + 3 t\,  \| h \|_{Ces(\infty)}\nonumber \\
&\geq&
\| (f - h)  \chi_{[0, \tau_1(t)]} \|_{Ces(1)} + 3 t\, \| h  \chi_{[0, \tau_1(t)]} \|_{Ces(\infty)}\nonumber \\
&=&
\| f  \chi_{[0, \tau_1(t)]} \|_{Ces(1)} - \| h  \chi_{[0, \tau_1(t)]} \|_{Ces(1)}\nonumber \\
&+&
3 t\,  \| h  \chi_{[0, \tau_1(t)]} \|_{Ces(\infty)}.
\label{equation22}
\end{eqnarray}
\vspace{-3mm}
Let us show that for any function $v \in Ces_{\infty},$ $v\ge 0,$ with ${\rm supp} \, v \subset [0, \tau_1(t)]$ we have
\vspace{2mm}
\begin{equation} \label{equation23}
\| v \|_{Ces(1)} \leq 3 t \, \| v \|_{Ces(\infty)}.
\end{equation}
In fact, by the assumption on the support of $v$ and by the Fubini theorem, we obtain
\begin{eqnarray*}
\| v \|_{Ces(1)}
&=&
\int_0^{\tau_1(t)} \Big( \frac{1}{s} \int_0^s v(u)\, du\Big)\, ds + \int_{\tau_1(t)}^1\Big(\frac{1}{s} \int_0^{\tau_1(t)} v(u)\, du\Big)\, ds\\
&=&
\int_0^{\tau_1(t)} \Big(\frac{1}{s} \int_0^s v(u)\, du\Big)\, ds + \int_0^{\tau_1(t)} \Big(\int_{\tau_1(t)}^1
\frac{1}{s}\, ds\Big) v(u)\, du \\
&=&
\int_0^{\tau_1(t)} \Big(\frac{1}{s} \int_0^s v(u)\, du\Big)\, ds + \int_0^{\tau_1(t)} v(u)\, du \, \ln \frac{1}{\tau_1(t)}.
\end{eqnarray*}
Since $\tau_1(t) \leq t$ it follows that
\begin{equation*}
\int_0^{\tau_1(t)} \Big(\frac{1}{s} \int_0^s v(u)\, du\Big)\, ds \leq \tau_1(t) \, \sup_{0 < s \leq \tau_1(t)} \frac{1}{s} \int_0^s v(u)\, du \leq t \, \| v \|_{Ces(\infty)}.
\end{equation*}
Moreover,
\begin{eqnarray*}
\int_0^{\tau_1(t)} v(u)\, du \, \ln \frac{1}{\tau_1(t)}
&\leq&
\tau_1(t) \ln \frac{1}{\tau_1(t)} \, \sup_{0 < s \leq \tau_1(t)} \frac{1}{s} \int_0^s v(u) \, du \\
&=&
\frac{\ln \frac{1}{t} + \ln \ln \frac{e}{t}}{\ln\frac{e}{t}} \, t \,  \| v \|_{Ces(\infty)}
\leq 2 t \, \| v \|_{Ces(\infty)},
\end{eqnarray*}
and estimate (\ref{equation23}) follows. Combining this estimate for $v=h\chi_{[0,\tau_1(t)]}$
together with \eqref{equation22} we conclude that
\begin{equation*}
3 \, ( \| g \|_{Ces(1)} + t \| h \|_{Ces(\infty)} ) \geq \| f  \chi_{[0, \tau_1(t)]} \|_{Ces(1)}.
\end{equation*}
Taking the infimum over all decompositions $f = g + h, g \in Ces_1, h \in Ces_{\infty}$ with $0 \leq g \leq f,
0 \leq h \leq f$ we obtain estimate (\ref{21}).

Next, since $Ces_1=L_1(\ln\frac{1}{s})$ and $Ces_{\infty}  \stackrel {1} \hookrightarrow L_1,$ we have
\begin{equation*}
K(t, f; L_1(\ln\frac{1}{s}), L_1) = K(t, f; Ces_1, L_1) \leq K(t, f).
\end{equation*}
Therefore, applying the well-known equality
\begin{equation*}
K(t, f; L_1(\ln\frac{1}{s}), L_1) = \int_0^1 \min(\ln\frac{1}{s}, t) |f(s)| \, ds
\end{equation*}
and the elementary inequality
\begin{equation*}
\int_0^1 \min(\ln\frac{1}{s}, t) |f(s)| \, ds \geq \int_{e^{-t}}^1 \ln\frac{1}{s} |f(s)| \, ds = \| f \chi_{[\tau_2(t), 1]} \|_{Ces(1)},
\end{equation*}
we obtain
\begin{equation*}
K(t, f) \geq  \| f \chi_{[\tau_2(t), 1]} \|_{Ces(1)}.
\end{equation*}
Inequality (\ref{19}) is an immediate consequence of the last inequality and estimate (\ref{21}). The proof of (i) is complete.

(ii) Since inequality \eqref{20} is obvious for $t\in [t_0,1]$, it
can be assumed that $0<t<t_0.$  Let again $f\in Ces_1,$ $f = g + h$, where $g \in Ces_1, h \in Ces_{\infty}$
and $0 \leq g \leq f, 0 \leq h \leq f$. Then for any $c \in (0, 1)$ we have
\begin{eqnarray}
 \| g \|_{Ces(1)} + t \| h \|_{Ces(\infty)}
&\geq&
\| g \chi_{[\tau_1(t), \tau_2(t)]} \|_{Ces(1)} + c\, t\,  \| (f-g) \chi_{[\tau_1(t), \tau_2(t)]} \|_{Ces(\infty)}\nonumber \\
&\geq&
\| g \chi_{[\tau_1(t), \tau_2(t)]} \|_{Ces(1)} -c\, t\, \| g  \chi_{[\tau_1(t), \tau_2(t)]} \|_{Ces(\infty)}\nonumber  \\
&+&
c\, t\, \| f \chi_{[\tau_1(t), \tau_2(t)]} \|_{Ces(\infty)}.
 \label{24}
\end{eqnarray}
We want to show that for every positive function $w \in Ces_1$ with ${\rm supp} \, w \subset [\tau_1(t), \tau_2(t)]$
the following inequality holds:
\begin{equation} \label{25}
\frac{1}{e^2}\, t \, \| w \|_{Ces(\infty)} \leq \| w \|_{Ces(1)} ~~ {\rm for ~any} ~~ 0 < t < t_0.
\end{equation}
Since
\begin{eqnarray*}
\| w \|_{Ces(1)}
&=&
\int_0^1 \frac{1}{s} \left( \int_{\tau_1(t)}^s w(u)\, du \cdot \chi_{[\tau_1(t), \tau_2(t)]}(s) + \int_{\tau_1(t)}^{\tau_2(t)} w(u)\, du \, \chi_{[\tau_2(t),1]}(s) \right) ds\\
&=&
\int_{\tau_1(t)}^{\tau_2(t)} \Big( \frac{1}{s} \int_{\tau_1(t)}^s w(u)\, du\Big) \, ds + \int_{\tau_1(t)}^{\tau_2(t)} w(u)\, du \,
\int_{\tau_2(t)}^1 \frac{ds}{s}\\
&=&
\int_{\tau_1(t)}^{\tau_2(t)} \Big( \int_u^{\tau_2(t)} \frac{ds}{s}\Big) \, w(u) \, du + \int_{\tau_1(t)}^{\tau_2(t)} w(u)\, du \, \ln \frac{1}{{\tau_2(t)}} \\
&=&
\int_{\tau_1(t)}^{\tau_2(t)} w(u)\, \ln \frac{\tau_2(t)}{u}\, du + t\,\int_{\tau_1(t)}^{\tau_2(t)} w(u)\, du,
\end{eqnarray*}
for proving the previous inequality it suffices to show that for all $t \in (0, t_0)$ and $s\in [\tau_1(t),\tau_2(t)]$ we have
\begin{equation} \label{26}
\frac{1}{e^2}\, t \int_{\tau_1(t)}^s w(u) \, du \leq s \Big( \int_{\tau_1(t)}^{\tau_2(t)} w(u)\, \ln \frac{\tau_2(t)}{u}\, du +
t\, \int_{\tau_1(t)}^{\tau_2(t)} w(u)\, du \Big).
\end{equation}
We consider the cases when $s \in [\tau_1(t), \frac{\tau_2(t)}{e}]$ and $s \in (\frac{\tau_2(t)}{e}, \tau_2(t)]$ separately.
Define a unique $t_1 \in (0, t_0)$ such that $\tau_1(t_1) =  \frac{\tau_2(t_1)}{e}$ and note that the segment
$[\tau_1(t), \frac{\tau_2(t)}{e}]$ is non-empty only if $0<t\le t_1.$ Let
$$
\varphi(s):= s \cdot \ln \frac{\tau_2(t)}{s} ~~{\rm for } ~~ s \in [\tau_1(t), \frac{\tau_2(t)}{e}].
$$
Since $\varphi^{\prime}(s) = \ln \frac{\tau_2(t)}{s} -1 = \ln \frac{\tau_2(t)}{e s} \geq 0$ for all $s \in [\tau_1(t), \frac{\tau_2(t)}{e}]$ it follows that $\varphi$ increases. Therefore, $\varphi(s) \geq \varphi(\tau_1(t))$ for all $s \in [\tau_1(t), \frac{\tau_2(t)}{e}]$ and so
\begin{equation} \label{27}
s \, \int_{\tau_1(t)}^{\tau_2(t)} w(u)\, \ln \frac{\tau_2(t)}{u}\, du \geq s \,  \ln \frac{\tau_2(t)}{s} \, \int_{\tau_1(t)}^s w(u) \, du
\geq \tau_1(t) \,  \ln \frac{\tau_2(t)}{\tau_1(t)} \, \int_{\tau_1(t)}^s w(u) \, du.
\end{equation}
We show that
\begin{equation}
\label{28}
\tau_1(t) \ln \frac{\tau_2(t)}{\tau_1(t)} \geq \frac{1}{e^2} \, t ~~ {\rm for ~ all } ~~ 0 < t \leq t_1.
\end{equation}
The function
$$
\psi(t) = \frac{\tau_1(t)}{t} \ln \frac{\tau_2(t)}{\tau_1(t)} = \frac{\ln \frac{\ln\frac{e}{t}}{t} - t}{\ln\frac{e}{t}} ~~ {\rm for} ~~ t \in (0, t_1]
$$
has the derivative $\psi^{\prime}(t) = - [(t+1)(1+\ln \frac{e}{t})+ \ln \tau_1(t)]/[t (\ln \frac{e}{t})^2]$. It is not hard to check
that $\psi$ is increasing
on $(0, t_2)$ and decreasing on $(t_2, t_1]$, where a unique $t_2 \in (0, t_1)$. Hence, by the definition of $t_1$, for all
$t\in (0,t_1]$, we have
$$
\psi(t) \geq \min[\psi(0^+), \psi(t_1)] = \min\left(1, \ln^{-1} \frac{e}{t_1}\right) =\ln^{-1} \frac{e}{t_1}= t_1^{-1}e^{-1-t_1}\ge e^{-2}.
$$
Thus, we obtain inequality (\ref{28}). Combining it with estimate \eqref{27}, we obtain \eqref{26} in the case when
$0<t\le t_1$ and $s \in [\tau_1(t), \frac{\tau_2(t)}{e}]$.

In the second case, when $s \in (\frac{\tau_2(t)}{e}, \tau_2(t)]$, we have $s \geq e^{-1-t} \geq e^{-2}$ and so
$$
t \, \int_{\tau_1(t)}^s w(u)\, du \leq e^2 \, t \, s \, \int_{\tau_1(t)}^{\tau_2(t)} w(u)\, du.
$$
Hence, estimate (\ref{26}) holds again, and so inequality \eqref{25} is proved.
Combining (\ref{25}) and (\ref{24}) with $c=e^{-2}$, we obtain
$$
\| g \|_{Ces(1)} + t \| h \|_{Ces(\infty)} \geq \frac{1}{e^2} \, t \, \| f_1 \|_{Ces(\infty)} ~~ {\rm for ~all} ~~ 0 < t < t_0.
$$
Taking the infimum over all decompositions $f = g + h, g \in Ces_1, h \in Ces_{\infty}$ with $0 \leq g \leq f,
0 \leq h \leq f$ we come to estimate (\ref{20}), and the proof of (ii) is complete.
\end{proof}

\begin{theorem} \label{Thm3} For every function $f \in Ces_1[0, 1]$ we have
\begin{eqnarray*}
\frac{1}{2e^2} \, ( \| f \chi_{[0, \tau_1(t)] \cup [\tau_2(t), 1]} \|_{Ces(1)}
&+&
t \, \| f \chi_{[\tau_1(t), \tau_2(t)]} \|_{Ces(\infty)} ) \\
&\leq&
K(t, f; Ces_1, Ces_{\infty}) \\
&\leq&
 \| f \chi_{[0, \tau_1(t)] \cup [\tau_2(t), 1]} \|_{Ces(1)} + t \, \| f \chi_{[\tau_1(t), \tau_2(t)]} \|_{Ces(\infty)},
\end{eqnarray*}
for all $0 < t < 1$, and
$K(t, f; Ces_1, Ces_{\infty}) =  \| f \|_{Ces(1)}$ for all $t \geq 1$.
\end{theorem}

\begin{proof} The first inequality is a consequence of Lemma \ref{lemma1} and
the definition of the $K$-functional. The equality
$K(t, f; Ces_1, Ces_{\infty}) = \| f \|_{Ces(1)} $ $(t \geq 1)$ follows from
the imbedding $Ces_{\infty}  \stackrel {1} \hookrightarrow Ces_1.$
\end{proof}

If a positive function $f \in Ces_1[0, 1]$ is decreasing, then the description of the $K$-functional can be simplified.

\begin{theorem} \label{Thm4} If $f \in Ces_1[0, 1], f \geq 0$ and $f$ is decreasing, then
\begin{equation} \label{29}
\frac{1}{3}\, \| f  \chi_{[0, \tau_1(t)]} \|_{Ces(1)} \leq K(t, f; Ces_1, Ces_{\infty}) \leq  \| f  \chi_{[0, \tau_1(t)]} \|_{Ces(1)}
\end{equation}
for all $0 < t  <1$ and $K(t, f; Ces_1, Ces_{\infty}) = \| f \|_{Ces(1)}$ for all $t\ge 1$.
\end{theorem}

\begin{proof} Taking into account the proof of Lemma 1(i) (see inequality (\ref{21})) it suffices to prove only the right-hand side
inequality in (\ref{29}).

Let $f_0: =  [ f - f(\tau_1(t)) ] \chi_{[0, \tau_1(t)]}$ and $f_1: = f - f_0$. Since $f \geq 0$ is decreasing, we have
$\| f_1\|_{Ces(\infty)} = f(\tau_1(t)).$  Therefore, by Fubini theorem,
$$
\| f_0 \|_{Ces(1)} + t\, \| f_1 \|_{Ces(\infty)} =
\int_0^1 \frac{1}{s} \int_0^s \left( f(u) - f(\tau_1(t)) \right)  \chi_{[0, \tau_1(t)]}(u) \, du \, ds + t \, f(\tau_1(t))
$$
$$
= \int_0^1 \frac{1}{s} \int_0^s f(u) \chi_{[0, \tau_1(t)]}(u) \, du \, ds - f(\tau_1(t)) \int_0^{\tau_1(t)} \ln \frac{1}{u} \, du + t \, f(\tau_1(t))
$$

\begin{eqnarray*}
&=&
\|f  \chi_{[0, \tau_1(t)]} \|_{Ces(1)}
-f(\tau_1(t)) \, \tau_1(t) \, \left[1 + \ln \frac{1}{\tau_1(t)} \right] + t \, f(\tau_1(t))\\
&=&
\|f  \chi_{[0, \tau_1(t)]} \|_{Ces(1)} + t \, f(\tau_1(t)) \left [ 1 - \frac{1 + \ln(\ln \frac{e}{t}/t)}{\ln \frac{e}{t}} \right]\\
&=&
\|f  \chi_{[0, \tau_1(t)]} \|_{Ces(1)} - \frac{t f(\tau_1(t))\ln(\ln \frac{e}{t})}{\ln \frac{e}{t}}\le \|f  \chi_{[0, \tau_1(t)]} \|_{Ces(1)},
\end{eqnarray*}
whence
\begin{equation*}
K(t, f; Ces_1, Ces_{\infty}) \leq \| f  \chi_{[0, \tau_1(t)]} \|_{Ces(1)},
\end{equation*}
and the desired result is proved.
\end{proof}

\vspace{3 mm}
\begin{center}
{\bf 6. Indentification of the real interpolation spaces $(Ces_1[0, 1], Ces_{\infty}[0, 1])_{1-1/p, p}$ for $1 < p < \infty$}
\end{center}

Let us define the weighted Ces\`aro function space $ Ces_p(\ln \frac{e}{t})[0, 1]$ consisting
of all Lebesgue measurable functions $f$ on $[0, 1]$ such that
\begin{equation*}
\| f \| _{Ces(p, \ln)}: = \Big( \int_0^1 \Big(\frac{1}{x} \int_0^x |f(t)| \, dt\Big)^p \, \ln \frac{e}{x} \, dx \Big)^{1/p} < \infty.
\end{equation*}
Clearly, $ Ces_p(\ln \frac{e}{t})[0, 1]  \stackrel {1} \hookrightarrow Ces_p[0, 1]$ for every $1<p<\infty,$ and
this imbedding is strict.

\begin{theorem} \label{Thm5} For $1 < p < \infty$
\begin{equation} \label{30}
(Ces_1[0, 1], Ces_{\infty}[0, 1])_{1-1/p, p} = Ces_p(\ln \frac{e}{t})[0, 1].
\end{equation}
\end{theorem}

\begin{proof}
Denote $X_p = (Ces_1, Ces_{\infty})_{1-1/p, p}, 1 < p < \infty$. Using Theorem \ref{Thm3} on the $K$-functional
for the couple $(Ces_1, Ces_{\infty})$ on $[0, 1],$ we have
\begin{eqnarray*}
\| f \|_{X_p}
&\leq&
\Big[ \int_0^{t_0} t^{-p} \| f \chi_{[0, \tau_1(t)]}\|_{Ces(1)}^p \, dt \Big]^{1/p} + \Big[ \int_0^{t_0} t^{-p} \| f \chi_{[\tau_2(t), 1]}\|_{Ces(1)}^p \, dt \Big]^{1/p} \\
&+&
\Big[ \int_0^{t_0} t^{-p} (t \| f  \chi_{[\tau_1(t), \tau_2(t)]}\|_{Ces(\infty)})^p \, dt \Big]^{1/p} + \Big[ \int_{t_0}^{\infty} t^{-p} \| f \|_{Ces(1)}^p \, dt \Big]^{1/p} \\
&=&
I_1 + I_2 + I_3 + I_4,
\end{eqnarray*}
where
\begin{eqnarray*}
I_1 &=&
\Big[ \int_0^{t_0} t^{-p} \Big( \int_0^{\tau_1(t)} Cf(s) \, ds + \int_{\tau_1(t)}^1 C(f \chi_{[0, \tau_1(t)]})(s)\, ds \Big)^p \, dt  \Big]^{1/p}\\
&\leq&
\Big[ \int_0^{t_0} t^{-p} \Big( \int_0^{\tau_1(t)} Cf(s) \, ds\Big)^p \, dt  \Big]^{1/p} + \Big[ \int_0^{t_0} t^{-p}
\Big( \int_{\tau_1(t)}^1 C(f \chi_{[0, \tau_1(t)]})(s)\, ds\Big)^p \, dt  \Big]^{1/p}\\
&=&
I_{11} + I_{12}.
\end{eqnarray*}

First of all, we estimate all five integrals from above. Since $\tau_1^{\prime}(t) = (\ln \frac{e}{t} + 1)/(\ln \frac{e}{t})^2$
and so $1/\ln(e/t) \leq \tau_1^{\prime}(t) \leq 2/\ln(e/t)$
for all $0 < t \leq 1$, we get
\begin{eqnarray*}
I_{11}^p
&\leq&
\int_0^{t_0} t^{-p} (\ln \frac{e}{t})^{p-1} \Big( \int_0^{\tau_1(t)} Cf(s) \, ds\Big)^p \, dt  \\
&\leq&
\int_0^{t_0} \tau_1(t)^{-p} \Big( \int_0^{\tau_1(t)} Cf(s) \, ds\Big)^p \, d \tau_1(t).
\end{eqnarray*}
Putting $u = \tau_1(t)$ and using classical Hardy inequality, we obtain
\begin{eqnarray*}
I_{11}
&\leq&
\Big[ \int_0^{\tau_1(t_0)} \Big(\frac{1}{u} \int_0^u Cf(s) \, ds\Big)^p \, du \Big]^{1/p} \leq \| C^2 f \|_{L_p[0, 1]} \\
&\leq&
p' \, \| C f \|_{L_p[0, 1]} = p'\, \| f \|_{Ces(p)} \leq p'\, \| f \|_{Ces(p, \ln)}.
\end{eqnarray*}
Next, by the estimate $\ln \frac{1}{\tau_1(t)} \leq 2 \ln \frac{e}{t},$ $0 < t \leq 1,$ we get
\begin{eqnarray*}
I_{12}^p
&=&
\int_0^{t_0} t^{-p} \Big( \int_{\tau_1(t)}^1 \Big( \frac{1}{s} \int_0^{\tau_1(t)} |f(u)| \, du\Big) \, ds \Big)^p \, dt \\
&=&
\int_0^{t_0} t^{-p} \Big( \int_0^{\tau_1(t)} |f(u)| \, du\Big)^p \, \ln^p \frac{1}{\tau_1(t)} \, dt \\
&\leq&
2^p \, \int_0^{t_0} \tau_1(t)^{-p} \Big( \int_0^{\tau_1(t)} |f(u)| \, du\Big)^p  \, dt.
\end{eqnarray*}
Substitution $t = \tau_1^{-1}(s)$ and the inequalities
\begin{equation} \label{31}
(\tau_1^{-1})^{\prime}(s) = \frac{1}{\tau_1^{\prime}(\tau_1^{-1}(s))} \leq \ln \frac{e}{\tau_1^{-1}(s)} \leq \ln \frac{e}{s}
\end{equation}
show that
\begin{eqnarray*}
I_{12}
&\le&
2 \, \Big [  \int_0^{\tau_1(t_0)} \Big(\frac{1}{s} \int_0^s |f(u)| \, du\Big)^p \ln \frac{e}{s} \, ds \Big ]^{1/p} \\
&\leq&
2 \, \Big [  \int_0^1 (Cf(s))^p \ln \frac{e}{s} \, ds \Big ]^{1/p}  = 2 \, \| f \|_{Ces(p, \ln)}.
\end{eqnarray*}
From the equality $Ces_1[0, 1] = L_1(\ln 1/u)$ and the inequalities $\ln 1/u \leq e (1 - u)$ $(1/e\le u\le 1)$ and
$\tau_2(t) = e^{-t} \geq 1 - t$ $(0 < t \leq 1)$  it follows
\begin{eqnarray*}
I_2^p
&=&
\int_0^{t_0} t^{-p} \| f \chi_{[\tau_2(t), 1]} \|_{Ces(1)}^p \, dt =
\int_0^{t_0} t^{-p} \Big( \int_{\tau_2(t)}^1 |f(u)| \ln \frac{1}{u} \, du\Big)^p \, dt \\
&\leq&
e^p \, \int_0^{t_0} t^{-p} \Big( \int_{\tau_2(t)}^1 |f(u)| (1 - u) \, du\Big)^p \, dt \\
&\leq&
e^p \, \int_0^{t_0} t^{-p} \Big( \int_{1 - t}^1 |f(u)| (1 - u) \, du\Big)^p \, dt.
\end{eqnarray*}
Arguing in the same way as in the second part of the proof of Theorem
\ref{Thm2}, for $g = f \chi_{[e^{-1}, 1]}$ and $0 < s \leq 1$ we have
$$
Cg(s) = \frac{1}{s} \int_{e^{-1}}^s |f(u)| \, du \geq  \int_{e^{-1}}^s |f(u)| \, du,
$$
whence $(Cg)^*(s) \geq  \int_{e^{-1}}^{1-s} |f(u)| \, du$ and
\begin{eqnarray*}
\int_0^t (Cg)^*(s) \, ds
&\geq&
\int_0^t \Big(\int_{e^{-1}}^{1-s} |f(u)| \, du\Big) \, ds =
\int_{e^{-1}}^{1 - t} \Big( \int_0^t |f(u)| \, ds\Big) \, du\\
&+& \int_{1-t}^{1} \Big( \int_0^{1-u} |f(u)| \, ds\Big) \, du
\geq  \int_{1-t}^{1}  |f(u)| (1 - u) \, du.
\end{eqnarray*}
Therefore, again by the classical Hardy inequality,
\begin{eqnarray*}
I_2
&\leq&
e \, \Big[ \int_0^{t_0} t^{-p} \Big(\int_0^t (Cg)^*(s) \, ds\Big)^p \, dt \Big]^{1/p} \\
&\leq&
e \, || C[(Cg)^*] \|_{L_p[0, 1]} \leq e p'\| (Cg)^*\|_{L_p[0, 1]} \\
&=&
e p'\| Cg \|_{L_p[0, 1]} = e p'\, \| f \chi_{[e^{-1}, 1]} \|_{Ces(p)} \\
&\leq&
e p' \, \| f \|_{Ces(p)} \leq e p'\, \| f \|_{Ces(p, \ln)}.
\end{eqnarray*}
For the third integral, we have
\begin{eqnarray*}
I_3
&=&
\Big[ \int_0^{t_0}  \| f  \chi_{[\tau_1(t), \tau_2(t)]}\|_{Ces(\infty)}^p \, dt \Big]^{1/p} \\
&\leq &
\Big[ \int_0^{t_0} \sup_{\tau_1(t) < s \leq 1/2} \Big( \frac{1}{s} \int_0^s |f(u) \chi_{[\tau_1(t), \tau_2(t)]}(u)| \, du \Big)^p \, dt \Big]^{1/p} \\
&+&
\Big[ \int_0^{t_0} \sup_{1/2 < s \leq \tau_2(t)} \Big( \frac{1}{s} \int_0^s |f(u) \chi_{[\tau_1(t), \tau_2(t)]}(u)| \, du \Big)^p \, dt \Big]^{1/p} \\
&=&
\Big[ \int_0^{t_0} \sup_{\tau_1(t) < s \leq 1/2} \Big( \frac{1}{s} \int_{\tau_1(t)}^s |f(u) | \, du \Big)^p \, dt \Big]^{1/p} \\
&+&
\Big[ \int_0^{t_0} \sup_{1/2 < s \leq \tau_2(t)} \Big( \frac{1}{s} \int_{\tau_1(t)}^s |f(u)| \, du \Big)^p \, dt \Big]^{1/p} = I_{31} + I_{32}.
\end{eqnarray*}
If $\tau_1(t) < s \leq 1/2$, then $2s \leq 1$ and
\begin{eqnarray*}
\int_{\tau_1(t)}^{2s} \Big(\frac{1}{v} \int_0^v |f(u)| \, du\Big)\, dv
&=&
\int_0^{\tau_1(t)} \Big( \int_{\tau_1(t)}^{2s} \frac{1}{v} \, dv\Big) \, |f(u)| \, du + \int_{\tau_1(t)}^{2s}
\Big( \int_u^{2s}  \frac{1}{v} \, dv\Big)\, |f(u)| \, du \\
&=&
\int_0^{\tau_1(t)} |f(u)| \, du \ln \frac{2 s}{\tau_1(t)} + \int_{\tau_1(t)}^{2s} |f(u)| \, \ln \frac{2s}{u} \, du \\
&\geq &
\frac{2s - \tau_1(t)}{2s} \, \int_{\tau_1(t)}^{2s} |f(u)| \, \ln \frac{2s}{u} \, du\\
&\geq&
\ln 2 \, \frac{2s - \tau_1(t)}{2s} \, \int_{\tau_1(t)}^{s} |f(u)| \, du.
\end{eqnarray*}
Thus,
\begin{eqnarray*}
\sup_{\tau_1(t) < s \leq 1/2} \frac{1}{s}  \int_{\tau_1(t)}^{s} |f(u)| \, du
&\leq&
\frac{2}{\ln 2} \sup_{\tau_1(t) < s \leq 1/2} \frac{1}{2s - \tau_1(t)}  \int_{\tau_1(t)}^{2s} Cf(v) \, dv \\
&\leq&
\frac{2}{\ln 2} \, MCf(\tau_1(t)),
\end{eqnarray*}
where $M$ is the maximal Hardy-Littlewood operator on $[0, 1]$.
The above estimates show that
$$
I_{31} \leq \frac{2}{\ln 2} \Big ( \int_0^{t_0} MC f(\tau_1(t))^p \, dt \Big)^{1/p}.
$$
Using once again substitution $t = \tau_1^{-1}(s)$ and estimates \eqref{31}, we obtain
\begin{eqnarray*}
I_{31}
&\leq&
\frac{2}{\ln 2} \Big [ \int_0^{\tau_1(t_0)} [MC f(s)]^p \ln \frac{e}{s} \, ds \Big]^{1/p}
\leq \frac{2}{\ln 2} \, \| MCf \|_{L_p(\ln \frac{e}{s})}.
\end{eqnarray*}
We will show in the next lemma that the maximal operator $M$ is bounded in $L_p(\ln \frac{e}{s})[0,1]$ for $1 < p < \infty$,
which implies that for some constant $B_p \geq 1,$ which depends only on $p$, we have
$$
I_{31} \leq \frac{2B_p}{\ln 2} \, \| Cf \|_{L_p(\ln \frac{e}{s})} = \frac{2B_p}{\ln 2} \, \| f \|_{Ces(p, \ln)}.
$$
For the second part of the integral $I_3$ we estimate in the following way:
\begin{eqnarray*}
I_{32}^p
&=&
\int_0^{t_0} \sup_{1/2 < s \leq \tau_2(t)} \Big( \frac{1}{s} \int_{\tau_1(t)}^s |f(u)| \, du \Big)^p \, dt \leq
2^p \, \int_0^{t_0} \Big( \int_{\tau_1(t)}^{\tau_2(t)} |f(u)| \, du \Big)^p \, dt \\
&\leq&
2^p \, \int_0^{t_0} \Big(\frac{1}{\tau_2(t)} \int_0^{\tau_2(t)} |f(u)| \, du \Big)^p \, dt,
\end{eqnarray*}
and, changing variable $s = \tau_2(t) = e^{-t},$ we obtain
\begin{eqnarray*}
I_{32}
& \leq &
2 \, \Big[ \int_{e^{- t_0}}^1 \Big(\frac{1}{s} \int_0^s |f(u)| \, du\Big)^p \, \frac{ds}{s} \Big]^{1/p} \leq
2 e^{t_0/p} \, \Big( \int_0^1 Cf(s)^p \, ds\Big)^{1/p} \\
&\leq&
2 e \, \| f \|_{Ces(p)} \leq 2 e \, \| f \|_{Ces(p, \ln)}.
\end{eqnarray*}
Since $t_0>1/2$, for the last integral we have
$$
I_4 = \frac{1}{(p-1)^{1/p} t_0^{1-1/p}}\, \| f \|_{Ces(1)} \leq \frac{2}{p-1} \, \| f \|_{Ces(1)} \leq \frac{2}{p-1} \, \| f \|_{Ces(p, \ln)}.
$$
Finally, summing up the above estimates, we get
$\| f \|_{X_p} \leq C_p\| f \|_{Ces(p, \ln)},$
where $C_p$ depends only on $p.$
Thus, the imbedding $Ces(p, \ln) \hookrightarrow X_p$ is proved.

Now, we proceed with estimations from below. Firstly, by inequality \eqref{21}, we have
\begin{equation}\label{32}
\| f \|_{X_p}^p \geq
3^{-p} \, \int_0^{t_0} t^{-p} \| f \chi_{[0, \tau_1(t)]} \|_{Ces(1)}^p dt=3^{-p} \, I_1^p\ge 3^{-p} \, I_{12}^p.
\end{equation}
It is not hard to check that $ \ln \frac{1}{\tau_1(t)} = \ln \frac{\ln(e/t)}{t} \geq e^{-1} \, \ln \frac{e}{t} $
for $t \in (0, t_0].$ Therefore,
$$
I_{12}^p = \int_0^{t_0} t^{-p} \Big( \int_0^{\tau_1(t)} |f(u)| \, du\Big)^p \, \ln^p \frac{1}{\tau_1(t)} \, dt
\geq e^{-p} \, \int_0^{t_0} \tau_1(t)^{-p} \Big(\int_0^{\tau_1(t)} |f(u)| \, du\Big)^p \, dt.
$$
Since $\tau_1'(s)\le 2/\ln(e/s)$, $\tau_1^{-1}(s)\le s\ln(e/s)$ and $\ln \ln (e/s)\le e^{-1}\ln(e/s)$ $(0<s\le 1)$,
we have
\begin{eqnarray*}
(\tau_1^{-1})^{\prime}(s)
&=&
\frac{1}{\tau_1^{\prime}(\tau_1^{-1}(s))} \geq \frac{1}{2} \, \ln \frac{e}{\tau_1^{-1}(s)} \geq
\frac{1}{2} \, \ln \frac{e}{s \ln(e/s)} \\
&=&
\frac{1}{2} \, (\ln \frac{e}{s} - \ln \ln \frac{e}{s}) \geq \frac{1}{2} \, (1-\frac{1}{e}) \ln \frac{e}{s}.
\end{eqnarray*}
Hence, after substitution $t = \tau_1^{-1}(s),$ we obtain
$$
I_{12}^p \geq e^{-p} \frac{1}{2} (1 - \frac{1}{e})\, \int_0^{\tau_1(t_0)} \Big(\frac{1}{s} \int_0^s |f(u)| \, du\Big)^p
\,\ln \frac{e}{s} \, ds
\geq \frac{1}{4} \, e^{-p} \, \int_0^{\tau_1(t_0)} Cf(s)^p \ln \frac{e}{s} \, ds,
$$
and so, taking into account \eqref{32}, we get
$$
\| f \|_{X_p}^p \geq 4^{-1} \, (3e)^{-p} \,\int_0^{\tau_1(t_0)} Cf(s)^p \ln \frac{e}{s} \, ds.
$$
On the other hand, by the definition of $t_0$,
\begin{eqnarray*}
\int_{\tau_1(t_0)}^1 Cf(s)^p \ln \frac{e}{s} \, ds
&\leq&
\ln \frac{e}{\tau_1(t_0)} \, \int_{\tau_1(t_0)}^1 Cf(s)^p \, ds \leq (1+t_0) \| f\|_{Ces(p)}^p \\
&\leq&
2 \, \| f\|_{Ces(p)}^p \leq 2 \, \| f\|_{X_p}^p,
\end{eqnarray*}
where the last inequality follows from imbedding
\eqref{imbedding9} of Theorem \ref{Thm2}. Hence,
$$
\| f\|_{X_p} \geq 8^{-1/p} \, (3e)^{-1} \, \Big(\int_0^1 Cf(s)^p \ln \frac{e}{s} \, ds\Big)^{1/p}
\geq \frac{1}{72} \, \| f \|_{Ces(p, \ln)},
$$
and the imbedding $X_p \hookrightarrow Ces(p, \ln)$ is proved. Thus, the proof of Theorem \ref{Thm5}
will be finished if we show that the following lemma holds.
\end{proof}

\begin{lemma} \label{Lemma2} If $1 < p < \infty$, then the maximal Hardy-Littlewood operator  $M$ on $[0, 1]$ is
bounded in the weighted space $L_p(\ln \frac{e}{x})[0,1]=L_p([0, 1], \ln \frac{e}{x} dx).$
\end{lemma}

\vspace{-2mm}

\begin{proof} Muckenhoupt \cite[Theorem 2]{Mu} proved that the maximal operator $M$ on $[0, 1]$ is bounded  in
$L_p([0, 1], w(x) dx)$ if and only if the weight $w(x)$ satisfies the so-called {\it $A_p$-condition} on $[0, 1]$, that is,
$$
\sup_{(a, b) \subset [0, 1]} \Big(\frac{1}{b-a} \int_a^b w(x)\, dx\Big)\, \Big(\frac{1}{b-a} \int_a^b w(x)^{-1/(p-1)} \, dx\Big)^{p-1}
< \infty.$$
Therefore, it is enough to show that for all intervals $(a, b) \subset [0, 1]$ we have
\begin{equation} \label{33}
\int_a^b \ln \frac{e}{x} \, dx \, \Big( \int_a^b (\ln \frac{e}{x})^{- \frac{1}{p-1}} \, dx\Big)^{p-1} \leq 2\, (b-a)^p.
\end{equation}
Note that for $t \in (0, b)$
\vspace{-3mm}
$$
\int_t^b \ln \frac{e}{x} \, dx = b \ln \frac{e}{b} - t \ln \frac{e}{t} + b - t
$$
and
\vspace{-2mm}
\begin{eqnarray*}
\int_t^b (\ln \frac{e}{x})^{- \alpha} \, dx
&=&
b \, (\ln \frac{e}{b})^{- \alpha} - t \, (\ln \frac{e}{t})^{- \alpha} - \alpha \, \int_t^b (\ln \frac{e}{x})^{- \alpha -1} \, dx \\
&\leq&
b \, (\ln \frac{e}{b})^{- \alpha} - t \, (\ln \frac{e}{t})^{- \alpha},
\end{eqnarray*}
where $\alpha>0.$ Since the functions
$$
\varphi_1(t) = \frac{b \ln (e/b)- t \ln (e/t) + b - t}{b - t}\;\;\mbox{and}\;\;
\varphi_2(t) = \frac{b \, (\ln (e/b))^{- \alpha} - t \, (\ln (e/t))^{- \alpha}}{b - t}$$
are both decreasing on the interval $(0, b)$ for every $0<b\le 1$
it follows that $\max_{0 < t < b} \varphi_1(t) = \varphi_1(0^+) = \ln(e^2/b)$ and
$\max_{0 < t < b} \varphi_2(t) = \varphi_2(0^+) = \ln^{- \alpha}(e/b)$.
Therefore, setting $\alpha = \frac{1}{p-1}$, for arbitrary $0\le a < b\le1$ we have
$$
\frac{1}{(b-a)^p} \int_a^b \ln \frac{e}{x} \, dx \, \Big( \int_a^b (\ln \frac{e}{x})^{- \frac{1}{p-1}} \, dx \Big)^{p-1}
\leq \ln \frac{e^2}{b} \, \left(\left( \ln \left(\frac{e}{b}\right)\right)^{-1/(p-1)} \right)^{p-1}
=\frac{\ln (e^2/b)}{\ln (e/b)} \leq 2,
$$
and inequality (\ref{33}) is proved.
\end{proof}

\begin{center}
{\bf 7. $Ces_p[0,1],$ $1<p<\infty,$  is not an interpolation space between $Ces_1[0,1]$ and $Ces_\infty[0,1]$}
\end{center}

\noindent
We start with two lemmas (it is instructive to compare the result from the first of them with
imbedding \eqref{imbedding9}).

\begin{lemma} \label{Lemma3} If $1 < p < \infty$, then
\begin{equation} \label{34}
Ces_p[0, 1] \not \hookrightarrow (Ces_1[0, 1], Ces_{\infty}[0,1])_{1-1/p, \infty}.
\end{equation}
\end{lemma}

\begin{proof}  Let us consider the family of characteristic functions
$f_s = \chi_{[0, s]}, 0 < s < 1$. As we know (cf. Theorem \ref{Thm4}),
$$
K(t, f_s; Ces_1, Ces_{\infty}) \geq \frac{1}{3} \| f_s \chi_{[0, \tau_1(t)]} \|_{Ces(1)} ~~ {\rm for ~ all} ~~ t > 0.
$$
Since
\begin{eqnarray*}
\| f_s \chi_{[0, \tau_1(t)]} \|_{Ces(1)}
&=&
\| \chi_{[0, \min(s, \tau_1(t))]} \|_{Ces(1)} = \| \chi_{[0, \min(s, \tau_1(t))]} \|_{L_1(\ln1/s)} \\
&=&
\int_0^{\min(s, \tau_1(t))} \ln \frac{1}{s} \, ds = \min(s, \tau_1(t)) \, \Big( \ln \frac{1}{\min(s, \tau_1(t))} + 1\Big),
\end{eqnarray*}
it follows that for all $t$ such that $\tau_1(t) \leq s$ we have
$\| f_s \chi_{[0, \tau_1(t)]} \|_{Ces(1)} \geq \tau_1(t) \, \ln \frac{1}{ \tau_1(t)}$.
Therefore, using the inequality $\tau_1^{-1}(s)\le s\ln(e/s)$ once again,
for $0 < s < e^{-1}$ we obtain
\begin{eqnarray*}
\| f_s\|_{(Ces_1, Ces_{\infty})_{1-1/p, \infty}}
&=&
\sup_{t > 0} t^{1/p-1} K(t, f_s; Ces_1, Ces_{\infty}) \\
&\geq&
\frac{1}{3} \, \sup_{t > 0, \tau_1(t) \leq s} t^{1/p-1}\, \tau_1(t)\, \ln \frac{1}{\tau_1(t)} \\
&\geq&
\frac{1}{3} \, \left(\tau_1^{-1}(s)\right)^{1/p-1} s \ln \frac{1}{s} \geq \frac{1}{6} \, (s \ln \frac{e}{s})^{1/p-1} s \ln \frac{1}{s} \\
&\geq&
\frac{1}{6} \, s^{1/p} (\ln \frac{e}{s})^{1/p}.
\end{eqnarray*}
On the other hand,
\begin{eqnarray*}
\| f_s\|_{Ces_p}
&=&
\Big[ \int_0^s \Big(\frac{1}{u} \int_0^u \chi_{[0, s]}(v)\, dv\Big)^p \, du + \int_s^1  \Big(\frac{1}{u} \int_0^u \chi_{[0, s]}(v)\, dv\Big)^p
 \, du \Big]^{1/p}\\
&=&
\Big( s + s^p \int_s^1 u^{-p} \, du\Big)^{1/p} = \Big(s + \frac{s^p}{p-1} \, (s^{1-p} - 1) \Big)^{1/p} \\
&=&
\Big(\frac{p}{p-1}\, s - \frac{1}{p-1}\, s^p\Big)^{1/p} \leq (p')^{1/p} \, s^{1/p}.
\end{eqnarray*}
Therefore, for $0 < s <  e^{-1}$
$$
\frac{\| f_s\|_{(Ces_1, Ces_{\infty})_{1-1/p, \infty}}}{\| f_s\|_{Ces_p}} \geq \frac{\frac{1}{6} \,
s^{1/p} (\ln \frac{e}{s})^{1/p}}{(p')^{1/p} \, s^{1/p}} \geq \frac{1}{6 p'} \, (\ln \frac{e}{s})^{1/p},
$$
whence
$$
\sup_{0<s<1}\frac{\| f_s\|_{(Ces_1, Ces_{\infty})_{1-1/p, \infty}}}{\| f_s\|_{Ces_p}}=\infty,
$$
which shows that (\ref{34}) holds.
\end{proof}

Recall that the {\it characteristic function} $\varphi(s, t)$ of an exact interpolation functor $\cal F$
is defined by the equality ${\cal F} ( s {\mathbb R}, t {\mathbb R}) = \varphi (s, t) \, {\mathbb R}$ for all $s, t > 0$.
By the Aronszajn--Gagliardo theorem (see \cite[Theorem~2.5.1]{BL} or \cite[Theorem~2.3.15]{BK}),
for arbitrary Banach couple $(X_0,X_1)$ and for every Banach space $X\in Int(X_0,X_1)$ there
is an exact interpolation functor $\cal F$ such that ${\cal F}(X_0,X_1)=X.$

\vspace{1mm}
\begin{lemma} \label{Lemma4} Let $1 < p < \infty.$ Suppose that the Ces{\`a}ro space
$Ces_p[0, 1] \in Int(Ces_1[0, 1], \newline Ces_{\infty}[0,1])$ 
and $\cal F$ is an exact interpolation functor such that
\begin{equation} \label{35}
{\cal F} (Ces_1[0, 1], Ces_{\infty}[0,1]) = Ces_p[0, 1].
\end{equation}
Then the characteristic function $\varphi(1, t)$ of $\cal F$ is equivalent to $t^{1/p}$ for $0 < t \leq 1$.
\end{lemma}

\begin{proof} To simplify notation let us denote $V_p:=Ces_p|_{[1/2,1]}$ $(1 \leq p \leq \infty)$, that is,
$V_p$ is the subspace of $Ces_p[0,1]$, which consists of all functions $f$
such that ${\rm supp} \, f \subset [\frac{1}{2}, 1]$. Since $(V_1, V_{\infty})$ is a complemented couple
of the Banach couple $(Ces_1[0, 1], Ces_{\infty}[0, 1])$, by (\ref{35}) and equality in Remark 5, we obtain
\begin{equation} \label{36}
{\cal F} (V_1, V_{\infty}) = V_p = (V_1, V_\infty)_{1-1/p, p}.
\end{equation}
Consider the sequence of functions $g_k(t) = \chi_{[1- 2^{-k}, 1 - 2^{-k -1}]}(t), k = 1, 2, \ldots$ and the linear projection
$$
Pf(t) = \sum_{k=1}^{\infty} 2^{k+1} \int_{1- 2^{-k}}^{1 - 2^{-k -1}} f(s)\, ds\cdot g_k(t),\;\; f \in V_{\infty}.
$$
We have
\begin{eqnarray*}
\| Pf \|_{V_{\infty}}
&\leq&
2 \, \| Pf \|_{L_1|_{[1/2,1]}} \leq 2 \,  \sum_{k=1}^{\infty} 2^{k+1} \int_{1- 2^{-k}}^{1 - 2^{-k -1}} |f(s)|\, ds\cdot 2^{-k-1} \\
&=&
2 \, \| f \|_{L_1|_{[1/2,1]}} \leq 2 \, \| f \|_{V_{\infty}},
\end{eqnarray*}
and, since $1 - u \leq \ln (1/u) \leq 2(1-u)$ for $1/2 \leq u \leq 1$,
\begin{eqnarray*}
\| Pf \|_{V_1}
&\leq&
\sum_{k=1}^{\infty} 2^{k+1} \int_{1- 2^{-k}}^{1 - 2^{-k -1}} |f(s)|\, ds\cdot \int_{1- 2^{-k}}^{1 - 2^{-k -1}} \ln \frac{1}{t} \, dt \\
&\leq&
\sum_{k=1}^{\infty} 2^{k+2} \int_{1- 2^{-k}}^{1 - 2^{-k -1}} |f(s)|\, ds\cdot \int_{1- 2^{-k}}^{1 - 2^{-k -1}} (1-t) \, dt \\
&\leq&
\sum_{k=1}^{\infty} 2^{k+2} \cdot 2^{-2k-1} \cdot \int_{1- 2^{-k}}^{1 - 2^{-k -1}} |f(s)|\, ds\\
&\leq&
4 \, \sum_{k=1}^{\infty} \int_{1- 2^{-k}}^{1 - 2^{-k -1}} |f(s)| (1-s) \, ds \\
&\leq&
4 \, \sum_{k=1}^{\infty} \int_{1- 2^{-k}}^{1 - 2^{-k -1}} |f(s)| \ln\frac{1}{s} \, ds = 4 \, \| f \|_{L_1(\ln1/s)} = 4 \, \| f \|_{V_1}.
\end{eqnarray*}
Therefore, $P$ is a bounded linear projection from $V_{\infty}$ onto ${\rm Im} \, P_{{\large |}_{V_{\infty}}}$ and from $V_1$
onto ${\rm Im} \, P_{{\large |}_{V_1}}$. At the same time, it is easy to see that the sequence
$\{ 2^{k+1}\, g_k \}_{k=1}^{\infty}$ is equivalent in $V_{\infty}$ (resp. in $V_1$) to the standard basis in $l_1$
(resp. in $l_1(2^{-k})$). Hence, $(l_1, l_1(2^{-k}))$ is a complemented subcouple of the Banach couple
$(V_1, V_{\infty})$ and therefore, by (\ref{32}) and by the Baouendi-Goulaouic result
\cite[Theorem 1]{BG} (see also \cite[Theorem~1.17.1]{Tr}),
\begin{equation*}
{\cal F} (l_1, l_1(2^{-k})) = (l_1, l_1(2^{-k}))_{1-1/p, p}.
\end{equation*}
In particular, from the last relation it follows that
$$
{\cal F} ( {\mathbb R}, 2^{-k}\, {\mathbb R}) = ( {\mathbb R}, 2^{-k}\, {\mathbb R})_{1-1/p, p} = 2^{-k/p}{\mathbb R}
$$
uniformly in $k \in \mathbb N$. Since the characteristic function of any exact interpolation functor
is quasi-concave \cite[Proposition~2.3.10]{BK}, this implies the result.
\end{proof}

\begin{theorem} \label{Thm7} For any $1 < p < \infty$ the space $Ces_p[0,1]$ is not an interpolation space between
the spaces $Ces_1[0,1]$ and $Ces_\infty[0,1]$.
\end{theorem}
\begin{proof}
Assume that $Ces_p[0,1]$ is an interpolation space between $Ces_1[0,1]$ and $Ces_{\infty}[0,1]$.
Then there is an exact interpolation functor $\cal F$ such that equality \eqref{35} holds.
By Lemma \ref{Lemma4}, the characteristic function $\varphi(1, t)$ of $\cal F$ is equivalent to $t^{1/p}$ for $0 < t \leq 1$.
Therefore, for any Banach couple $(X_0, X_1)$ we have
\begin{equation} \label{37}
{\cal F} (X_0, X_1) \subset (X_0, X_1)_{\psi, \infty},
\end{equation}
where $(X_0, X_1)_{\psi, \infty}$ is the real interpolation space consisting of all $x \in X_0 + X_1$
such that
$\sup_{t > 0} \frac{\psi(t)}{t}\, K(t, x; X_0, X_1) < \infty$ and $\psi(t) = \min(1, t^{1/p})$
\cite[Proposition 3.8.6]{BK}. Since  $Ces_\infty[0, 1]  \stackrel {1} \hookrightarrow Ces_{1}[0, 1]$, then applying (\ref{37})
to the couple $(Ces_1[0,1], Ces_{\infty}[0,1])$, we obtain
\begin{equation} \label{38}
{\cal F} (Ces_1[0,1], Ces_{\infty}[0,1]) \subset (Ces_1[0,1], Ces_{\infty}[0,1])_{1-1/p, \infty},
\end{equation}
whence $Ces_p[0,1] \subset (Ces_1[0,1], Ces_{\infty}[0,1])_{1-1/p, \infty}$. But in view of Lemma \ref{Lemma3} the
last imbedding does not hold,  and the proof is complete.
\end{proof}


\vspace{3mm}

\noindent
{\footnotesize Department of Mathematics and Mechanics, Samara State University\\
Acad. Pavlova 1, 443011 Samara, Russia}\\
 ~{\it e-mail address:} {\tt astashkn@ssu.samara.ru} \\

\vspace{-3mm}

\noindent
{\footnotesize Department of Engineering Sciences and Mathematics\\
Lule\r{a} University of Technology, SE-971 87 Lule\r{a}, Sweden} \\
~{\it e-mail address:} {\tt lech.maligranda@ltu.se} \\

\end{document}